\documentclass[11pt]{amsart}
\usepackage{geometry}                
\usepackage{graphicx}
\usepackage{amssymb}
\usepackage{amsmath,amscd,latexsym,mathrsfs,amsfonts}
\usepackage{epstopdf}
\usepackage{amsthm}
\usepackage{caption}
\usepackage{tabularx} 
\usepackage[vlines]{tabularht}
\usepackage{adjustbox}

\newtheorem{thm}{Theorem}[section] 

\theoremstyle{definition}

\title{Numerical studies of Thompson's group $F$ and related groups}
\author{Andrew Elvey Price and Anthony J Guttmann}
\begin{document}

\begin{abstract}
We have developed polynomial-time algorithms to generate terms of the cogrowth series for groups $\mathbb{Z}\wr \mathbb{Z},$ the lamplighter group,    $(\mathbb{Z}\wr \mathbb{Z})\wr \mathbb{Z}$ and the Navas-Brin group $B.$ We have also given an improved algorithm for the coefficients of Thompson's group $F,$ giving 32 terms of the cogrowth series. We develop numerical techniques  to extract the asymptotics of these various cogrowth series. We present improved rigorous lower bounds on the growth-rate of the cogrowth series for Thompson's group $F$ using the method from \cite{HHR15} applied to our extended series. We also generalise their method by showing that it applies to loops on any locally finite graph. Unfortunately, lower bounds less than 16 do not help in
determining amenability.

Again for Thompson's group $F$ we prove that, if the group is amenable, there cannot be a sub-dominant stretched exponential term in the asymptotics\footnote{ }.  Yet the numerical data provides compelling evidence for the presence of such a term. This observation suggests a potential path to a proof of non-amenability: If the universality class of the cogrowth sequence can be determined rigorously, it will likely prove non-amenability.

We estimate the asymptotics of the cogrowth coefficients of $F$ to be
$$ c_n \sim c \cdot \mu^n \cdot \kappa^{n^\sigma \log^\delta{n}} \cdot n^g,$$ where $\mu \approx 15,$ $\kappa \approx 1/e,$ $\sigma \approx 1/2,$ $\delta \approx 1/2,$ and $g \approx -1.$ The growth constant $\mu$ must be 16 for amenability. These two approaches, plus a third based on extrapolating lower bounds,  support the conjecture \cite{ERvR15, HHR15} that the group is not amenable.

\end{abstract}

\maketitle
\section{Introduction}\label{sec:intro}

In an attempt to find compelling evidence for the amenability or otherwise of Thompson's group $F$, we have studied, numerically, the co-growth sequence of a number of infinite, finitely generated amenable groups whose asymptotics are, in most cases, partially or fully known. We have chosen a number of examples with increasingly complex asymptotics. Using the experience and insights gained from these examples, we turn to a study of Thompson's group $F$, having first developed an improved algorithm for the generation of the co-growth sequence, which we evaluate to O$(x^{32}).$

The cogrowth series of a group ${\mathcal G}$ with finite, inverse closed, generating set $S$ is $$C_{\mathcal G} = \sum_{n \ge 0} c_n x^n,$$ where $c_{n}$ is the number of words $w$ of length $2n$ over the alphabet $S,$ which satisfy $w=_{\mathcal G}1$ i.e. $w$ is the identity in the group ${\mathcal G}.$ There are many equivalent definitions of amenability. A standard one is that a group $G$ is amenable if it admits a left-invariant finitely additive probability measure $\mu.$ A consequence of the Grigorchuk-Cohen  \cite{G78, C82} theorem is that $G$ is amenable if and only if the radius of convergence of $C_{G}$ is $1/|S|^{2}$. In particular, Thompson's group $F$ amenable if and only if its cogrowth sequence has exponential growth rate 16.

We have developed new, polynomial-time algorithms to generate coefficients for the lamplighter group, and for general wreath product groups, $W_d = \mathbb{Z}\wr_d \mathbb{Z}.$ We also give a polynomial time algorithm for the cogrowth coefficients of the Navas-Brin group, and an improved algorithm to generate the coefficients of Thompson's group $F,$ generating the cogrowth sequences to O$(x^{128})$ and O$(x^{32})$ for $B$ and $F$ respectively.

The amenable group introduced independently by Navas \cite{N04} and Brin \cite{B05}, which we call the Navas-Brin group $B,$ is a subgroup of Thompson's  group $F$, and is defined as an infinite wreath product, with an extra generator which commutes each generator of the infinite wreath product to the next one. 
It has 2 generators, so the growth rate of the cogrowth sequence is 16. It also has a sub-exponential growth term that is very close to exponential, and so makes the growth rate difficult to estimate accurately with the number of terms at our disposal.

Using results of Pittet and Sallof-Coste \cite{PS-C00, PS-C02}, we prove that the cogrowth coefficients $c_n$ of Thompson's group $F$ satisfy $$c_n < 16^n\cdot \lambda^{-n^\kappa}$$ for any real numbers $\kappa < 1,$ and $\lambda >1.$ That is to say, if Thompson's group $F$ is amenable, then its asymptotics cannot contain a stretched-exponential term\footnote{We define {\em stretched exponential} more broadly than usual. It normally refers to a term of the form $e^{-t^\beta},$ with $t > 0$ and $0 < \beta < 1.$ We allow behaviour such as $e^{-t^\beta \cdot \log^\delta{t}},$  or indeed any appropriate logarithmic term. We do not have a name for sub-exponential growth of the form $e^{-t/\log^\delta{t}},$ with $\delta >0 $ (or appropriate logarithmic function) which is the type of term that must be present in the cogrowth series of the Navas-Brin group, and indeed in Thompson's group $F$ if it were amenable. }. Such a term is present in the asymptotics of the lamplighter group $L$ and the family of groups  $W_d.$  Furthermore, our numerical study reveals compelling evidence for the {\em presence} of such a term in the asymptotics of the coefficients of $F.$ This is our first strong evidence that Thompson's group $F$ is not amenable. Our second piece of evidence is the estimation of the growth constant. For amenability, the growth constant must be 16. We find that it is very close to 15.0 (we do not suggest it is exactly 15, but that is certainly a possibility).

Our numerical analysis relies on a number of methods that are well-known in the statistical mechanics and enumerative combinatorics community. Many are reviewed in \cite{G89} and \cite{GJ09}. For studies of the cogrowth asymptotics we primarily rely on the behaviour of the {\em ratio} of successive coefficients, as irrespective of the sub-dominant asymptotics, this ratio must go to the growth constant in the limit as the order of the coefficients goes to infinity. 

One new technique that we make use of in our study of the groups $B$ and $F$ is that of {\em series extension} \cite{G16}. In the case of group $B,$ we have 128 exact coefficients, but predict a further 590 ratios (and terms) with an estimated accuracy of, at worst, 1 part in $5 \times 10^{-7}.$ Having these extra (approximate) terms greatly improves the quality of the analysis we can perform. Similarly, for group $F,$ we use 32 exact terms to predict a further 200 ratios (and terms) with an estimated accuracy of 1 part in $4 \times 10^{-5}.$ This level of accuracy is more than sufficient for the graphical techniques we use to extract the asymptotics.

Another approach to estimating the growth rate was introduced by Haagerup, Haagerup and Ramirez-Solano in \cite{HHR15} who proved that the cogrowth sequence of Thompson's group  $F$  is given by the moments of a probability measure. We extend this to prove that this observation applies to the cogrowth sequence of any Cayley graph. In this way a sequence of rigorous lower bounds to the growth constant of the cogrowth series can be constructed. This approach also gives some stronger, non-rigorous, pseudo-bounds. Further details of this method, and some results, are given in section \ref{sec:moments}.

The simplest examples of groups we have chosen have asymptotics of the form $$c_n \sim c \cdot \mu^n \cdot n^g,$$ where $c$ is a constant, $\mu$ is the {\em growth constant} and $g$ is an exponent.

The first example of such a group is $\mathbb{Z}^2,$ which is a particularly simple case as both the coefficients and generating function are exactly known. In fact $c_n = {2n \choose n}^2,$ and the generating function $C_{\mathbb{Z}^2} =  2 {\bf K} \left ( \frac{4\sqrt{x}}{\pi} \right ),$ where ${\bf K} $ is the complete elliptic integral of the first kind.

The second example is the Heisenberg group, for which the asymptotic form of the coefficents is known \cite{G11} to be $c_n \sim 0.5 \cdot 16^n \cdot n^{-2},$ corresponding to a generating function $$C_{Heisenberg} \sim \frac{1}{2} (1-16x)\log(1-16x).$$ We have calculated 90 terms of the generating function, and show that this is sufficient to get a very precise asymptotic representation of the coefficients.

The next level of asymptotic complexity arises when there is an additional stretched-exponential term, so that the coefficients of the generating function behave as $$c_n \sim c \cdot \mu^n \cdot \kappa^{n^\sigma} \cdot n^g,$$ where $0 < \kappa < 1,$ and $0 < \sigma < 1.$ There is no known simple expression for the corresponding generating function in such cases\footnote{See, for example \cite{G15} for a discussion of this point, and further examples of such generating functions.}. The lamplighter group $L$ is the wreath product of the group of order two with the integers, $L=\mathbb{Z}_{2}\wr\mathbb{Z}.$  The growth rate is known, $\mu =9,$ and from Theorem 3.5 of \cite{PS-C02} it follows that $\sigma = 1/3,$ and from \cite{R03} we know that the exponent $g=1/6.$ So for the lamplighter group, $c_n \sim c \cdot 9^n \cdot \kappa^{n^{1/3}} \cdot n^{1/6}.$ Methods to extract the asymptotics from the coefficients have been developed, and are described in \cite{G15}. We give a polynomial time algorithm to generate the coefficients, and use it to determine the first 201 coefficients, from which we are able to estimate the correct values of the parameters $\mu,$ $ \sigma$ and $g.$

We next consider wreath products $W_d=\mathbb{Z}\wr_d \mathbb{Z}.$ In that case the exponent of the stretched-exponential term also includes a fractional power of a logarithm. Coefficients of the generating function behave as given by Theorem 3.11 in \cite{PS-C02}, so that $$c_n \sim c \cdot \mu^n \cdot \kappa^{n^\sigma \log^\delta{n}} \cdot n^g,$$ where $0 < \kappa < 1,$ and $0 < \sigma, \, \delta < 1.$ 

For $ d=1,$ one has $\mu=16,$ $ \sigma=1/3,$ $\delta=2/3$ and $g$ is not known. For $ d=2,$ one has, again by Theorem 3.11 in \cite{PS-C02}, $\mu=36,$ $ \sigma=1/2$ and $\delta=1/2.$ For general $ d,$ $ \mu=( 2d )^2,$ $\sigma= d/( d+2),$ and $\delta=2/( d+2).$ 

Note that this dimensional dependence of the exponent $\sigma$ of the stretched-exponential term appears to be a common feature among a broad class of problems. For example, if one considers the problem of a self-avoiding walk attached to a surface at its origin (or a Dyck path or a Motzkin path) and pushed toward the surface at its end-point (or its highest vertex), then, as shown in \cite{BGJL15} there is a stretched-exponential term in the asymptotics of the coefficients, with exponent $\sigma=1/(1+ d_f),$ where $ d_f$ is the fractal dimension of the walk/path. Whether this dimensional dependence is in fact a ubiquitous feature of such stretched-exponential terms remains an open question.

We have studied two examples, $W_1=\mathbb{Z}\wr \mathbb{Z}$ and $W_2=(\mathbb{Z}\wr \mathbb{Z}) \wr \mathbb{Z},$ based on the series we have generated of 276 and 133 terms respectively. We find that the presence of the confluent logarithmic term in the exponent makes the analysis significantly more difficult, but we can nevertheless accurately estimate the growth constant $\mu$ and less precisely estimate the sub-dominant growth rate $\kappa$ and the exponents $\sigma$ and $\delta$. Our estimates of the exponent $g$ are not precise enough to be useful.

We then turn to a contrived example, a constructed series with the asymptotics of $W_d=\mathbb{Z}\wr_d \mathbb{Z},$ with $ d=98.$ As $ d$ increases, the exponent in the stretched-exponential term gets closer to 1, and so this term behaves more and more like the dominant exponential growth term $\mu^n.$ We show that estimating the correct growth constant even approximately requires careful analysis, and appropriate techniques. This serves as a caution, and underlies that our conclusions regarding the non-amenability of Thompson's group $F$ assumes the absence of some unknown functional pathology.

Finally we study two groups whose behaviour is not fully known. The first is the Navas-Brin group $B.$  We give a polynomial-time algorithm to generate the coefficients, and in this way generate the first 128 terms, then use these to estimate the next 590 ratios.  This group has a sub-exponential growth term that is very close to exponential, and so makes the growth rate difficult to estimate accurately with the number of terms at our disposal. The second is Thompson's group $F$ where we have 32 exactly known terms, and 200 estimated ratios of terms.

The makeup of the paper is as follows. In Section \ref{sec:series} we describe the algorithms developed for the cogrowth series of the lamplighter group $L$, $W_1,$ $W_2,$  $B$ and Thompson's group $F.$ In Section \ref{sec:thom} we discuss the possible asymptotic form of the cogrowth series for Thompson's group $F,$ and prove the absence of a stretched-exponential term. In Section \ref{sec:moments} we develop the idea that the cogrowth coefficients can be represented as the sequence of moments of a probability measure. With this identification we establish rigorous lower-bounds on the growth constant for Thompson's group $F.$ In Section \ref{sec:analysis} we analyse the series expansions for the cogrowth series of all the groups we have mentioned above, apart from $B$ and $F.$ Section \ref{sec:extension} is devoted to a description of the method of series extension that we employ, and in  Sections \ref{sec:Brin} and \ref{sec:Thompson} we use this method and the techniques discussed in the previous section to analyse the Navas-Brin group $B$ and Thompson's group $F.$ Section \ref{sec:conclusion} comprises a discussion and conclusion.


\section{Series generation}\label{sec:series}
In this section we describe the algorithms we have used to compute the terms of the cogrowth sequence of various groups. We start by describing polynomial time algorithms which we have found and used for the groups $L$, $W_1,$ $W_2,$ and $B$. Finally we describe the algorithm which we have used for Thompson's group $F$. The first 50 coefficients for the group $B$ are given in Table \ref{tab:nb50},  while the coefficients of the cogrowth series of $F$ are given in Table \ref{tab:thomp}.

\subsection{Wreath Products $G \wr {\mathbb Z}$}
Let $G$ be a group with finite generating set $S$. We will describe a polynomial time algorithm for computing the cogrowth series of $G\wr\mathbb{Z}$, with respect to the generating set $\{a\}\cup S$, where $a$ generates $\mathbb{Z}$, given the corresponding series for $G$. In particular, this give a polynomial time algorithm to compute the cogrowth of the lamplighter group $\mathbb{Z}_{2}\wr\mathbb{Z}$ as well as groups such as $\mathbb{Z}\wr\mathbb{Z}$ and $( \mathbb{Z}\wr\mathbb{Z})\wr\mathbb{Z}$.

Let $a_{k}$ be the number of loops of length $k$ in $G$. For example, if $G=\mathbb{Z}$, then $a_{2k}={2k\choose k}$ and $a_{2k+1}=0$ for all $k\in\mathbb{Z}_{\geq0}$. Then for each positive integer $n$, define the generating function $P_{n}(x)$ by
\[P_{n}(x)=\sum_{j=0}^{\infty}{j+n-1\choose n-1}a_{j}x^j.\]
This is the generating function for $n$-tuples of words $w_{1},w_{2},\ldots,w_{n}\in S^{*}$ such that $\overline{w_{1}\ldots w_{n}}=1$, counted by the length of the word $w_{1}\ldots w_{n}$.

Given a loop $l$ in $G\wr\mathbb{Z}$, we define the base loop $l'$ of $l$ to be the loop in $\mathbb{Z}$ made up of only the terms $a$ and $a^{-1}$ in $l$. For each positive integer $i$, let $c_{i}$ be the number of steps in the baseloop $l'$ from $a^{i-1}$ to $a^{i}$ (which is the same as the number of steps from $a^{i}$ to $a^{i-1}$) and let $d_{i}$ be the number of steps from $a^{-i+1}$ to $a^{-i}$. Let $m$ and $n$ be maximal such that $c_m,d_n>0$. Then the length of $l'$ is equal to 
\[\sum_{i=1}^{m}2c_{i}+\sum_{j=1}^{n}2d_{j}.\]

Let $l'=a_{1}a_{2}\ldots a_{|l'|}$ and $l=w_{1}a_{1}w_{2}\ldots a_{|l'|}w_{|l'|+1}$, where each $w_{i}$ is a word in $(S\cup S^{-1})^{*}$. We say that the height of one of the subwords $w_{i}$ is equal to the integer $p$ which satisfies $a^{p}=a_{1}\ldots a_{i}$. Then $l$ is a loop if and only if for any height $h$, concatening all of the words $w_{i}$ at height $h$ creates a loop in $G$. Hence the generating function for the sections at height $h$ is $P_{r}(x)$ where $r$ is the number of these sections. If $h>0$ then $r=c_{h}+c_{h+1}$, if $h<0$ then $r=d_{-h}+d_{-h+1}$ and if $h=0$ then $r=c_1+d_1+1$. Hence,  by considering the sections of $l$ at each height separately, we see that the generating function for loops $l$ with base loop $l'$ is equal to
\begin{equation}\label{gengbase}x^{|l'|}P_{d_{n}}(x)P_{c_{m}}(x)P_{c_{1}+d_{1}+1}(x)\prod_{i=1}^{m-1}P_{c_{i}+c_{i+1}}(x)\prod_{j=1}^{n-1}P_{d_{j}+d_{j+1}}(x),\end{equation}
assuming that $m,n\geq1$. Similarly, if $m=0$ and $n\geq1$, the generating function is
\[x^{|l'|}P_{d_{n}}(x)P_{d_{1}+1}(x)\prod_{j=1}^{n-1}P_{d_{j}+d_{j+1}}(x).\]
If $n=0$ and $m\geq1$, the generating function is
\[x^{|l'|}P_{c_{m}}(x)P_{c_{1}+1}(x)\prod_{i=1}^{m-1}P_{c_{i}+c_{i+1}}(x).\]
Finally, if $m=0$ and $n=0$, then the generating function is $P_{1}(x).$
So we now need to sum this over all possible base loops $l'$.

For a given pair of sequences $c_{1},\ldots,c_{m},d_{1},\ldots,d_{n}$, the number of such base loops is equal to
\begin{equation}\label{bloops}{c_{1}+d_{1}\choose c_{1}}\prod_{i=1}^{m-1}{c_{i}+c_{i+1}-1\choose c_{i}-1}\prod_{j=1}^{n-1}{d_{j}+d_{j+1}-1\choose d_{j}-1}.\end{equation}
This is because from each vertex $i>0$ we can choose the order of the outgoing steps, except that the last one must be a left step, and there are $c_{i}-1$ other left steps and $c_{i+1}$ right steps. Hence there are ${c_{i}+c_{i+1}-1\choose c_{i}-1}$ possible orders of the steps leaving any vertex $i>0$, and similarly ${d_{j}+d_{j+1}-1\choose d_{j}-1}$ possible orders of the steps leaving any vertex $-j$ for $j>0$. Finally, there are ${c_{1}+d_{1}\choose c_{1}}$ possible orders of the steps leaving the vertex 0. It is easy to see that for any possible choice of these orders there is exactly one corresponding base loop $l'$.

Now using (\ref{gengbase}) and (\ref{bloops}) it follows that for any pair of sequences $c_{1},\ldots,c_{m},d_{1},\ldots,d_{n}$, with $m,n\geq1$, the generating function for the corresponding loops $l$ in $G\wr\mathbb{Z}$ is equal to 
\scriptsize
\begin{equation}\label{long}x^{2c_{1}+2d_{1}}{c_{1}+d_{1}\choose c_{1}}P_{d_{n}}P_{c_{m}}P_{c_{1}+d_{1}+1}\prod_{i=1}^{m-1}x^{2c_{i+1}}{c_{i}+c_{i+1}-1\choose c_{i}-1}P_{c_{i}+c_{i+1}}\prod_{j=1}^{n-1}x^{2d_{j+1}}{d_{j}+d_{j+1}-1\choose d_{j}-1}P_{d_{j}+d_{j+1}}.\end{equation}
\normalsize
If $m=0$ and $n\geq1$, the generating function is
\begin{equation}\label{notlong}x^{2d_{1}}P_{d_{n}}P_{d_{1}+1}\prod_{j=1}^{n-1}x^{2d_{j+1}}{d_{j}+d_{j+1}-1\choose d_{j}-1}P_{d_{j}+d_{j+1}}.\end{equation}
If $m\geq1$ and $n=0$ we get a similar generating function, and if $m=n=0$ we get $P_{1}(x)$.

To calculate these we define some new power series $\Omega_{d}(x)$ by
\[\Omega_{d}(x)=\sum P_{d_{n}}\prod_{j=1}^{n-1}x^{2d_{j+1}}{d_{j}+d_{j+1}-1\choose d_{j}-1}P_{d_{j}+d_{j+1}}(x),\]
where the sum is over all sequences $n,d_{1},d_{2},\ldots,d_{n}$ with $d_{1}=d$. Then it follows immediately from \eqref{long} and \eqref{notlong} that the generating function $F$ for the cogrowth series series of $G\wr\mathbb{Z}$ is given by
\begin{equation}\label{FinO}F(x)=\left(\sum_{c,d=1}^{\infty}x^{2c+2d}{c+d\choose c}P_{c+d+1}(x)\Omega_{d}(x)\Omega_{c}(x)\right)+2\left(\sum_{d=1}^{\infty}x^{2d}P_{d}(x)\Omega_{d}(x)\right)+P_{1}(x).\end{equation}
So now we just need to calculate $\Omega_{d}(x)$ for each positive integer $d$. First, the contribution to $\Omega_{d}$ from the case where $n=1$ is $P_{d_{n}}=P_{d_{1}}=P_{d}$. The contribution from the case where $n=1$ and $d_{2}=b$ for some fixed positive integer $b$ is
\[x^{2b}{d+b-1\choose d-1}P_{d+b}(x)\Omega_{b}(x).\]
Hence, we have the equation
\begin{equation}\label{Om}\Omega_{d}(x)=P_{d}(x)+\sum_{b=1}^{\infty}x^{2b}{b+d-1\choose d-1}P_{b+d}(x)\Omega_{b}(x).\end{equation}
Using this equation we can calculate the coefficient of $x^{k}$ in $\Omega_{d}$ of $x$ in terms of coefficients of $x^{j}$ in $\Omega_{b}(x)$ where we only need to consider $j,b$ satisfying $2b+j\leq k$ (hence $j\leq k-2$). This takes polynomial time using a simple dynamic program.

\subsection{The Navas-Brin group $B$} \label{subsec:nb}
In this section we adapt the previous algorithm to calculate the cogrowth series for the Navas-Brin group $B$. Again this is a polynomial time algorithm, however the polynomial has higher degree than the one for the previous section. The group $B$ is defined as the semi-direct product
\[\left(\ldots\wr\mathbb{Z}\wr\mathbb{Z}\wr\mathbb{Z}\wr\mathbb{Z}\wr\ldots\right)\rtimes\mathbb{Z},\]
where the copies of $\mathbb{Z}$ in the wreath product are generated by $\ldots, a_{2},a_{1},a_{0},a_{-1},a_{-2},\ldots$ and the generator $t$ of the other copy of $\mathbb{Z}$ satisfies $ta_{i}t^{-1}=a_{i+1}$ for each $i$. Note that the group $B$ is generated by the two elements $t$ and $a=a_{0}$. The group $B$ was described independently in \cite{N04} and on page 638 in \cite{B05}, where Brin showed that is an amenable supgroup of Thompson's group $F$. In that paper it is the group generated by $f$ and $h$.

We define the $t$-height of a word over the generating set $\{a,t,a^{-1},t^{-1}\}$ to be the sum of the powers of $t$. Before counting the total number of loops, we will count the number of loops where any initial subword has non-negative height. Let $G(x,y)$ be the generating function for these, where $x$ counts the total length and $y$ counts the number of steps of the loop which end at height 0. For each positive integer $n$, let $H_{n}(x,y)$ be the generating function for $n$-tuples $w_{1},w_{2},\ldots, w_{n}$ of words in $\{a,a^{-1},t,t^{-1}\}^{*}$ which each end at height 0 and which have no $a$ or $a^{-1}$ steps at height 0, such that $\overline{w_{1}\ldots w_{n}}=1$. In this generating function, $x$ counts the total length of $w_{1}\ldots w_{n}$ and $y$ counts the total number of steps which end at height 0. Given such a loop $l$, let the baseloop $l'$ be the subword consisting of all $a$ and $a^{-1}$ steps at $t$-height 0. Similarly to the previous algorithm, we let $c_{i}$ be the number of steps in $l'$ from $a^{i-1}$ to $a^{i}$, and $d_{i}$ be the number of steps in $l'$ from $a^{-i+1}$ to $a^{-i}$. Then the length $|l'|$ of $l'$ is equal to
\[\sum_{i=1}^{m}2c_{i}+\sum_{j=1}^{n}2d_{j}.\]
As in the previous subsection, for a given pair of sequences $c_{1},\ldots,c_{m},d_{1},\ldots,d_{n}$, the number of such base loops is equal to
\begin{equation}\label{bloops}{c_{1}+d_{1}\choose c_{1}}\prod_{i=1}^{m-1}{c_{i}+c_{i+1}-1\choose c_{i}-1}\prod_{j=1}^{n-1}{d_{j}+d_{j+1}-1\choose d_{j}-1}.\end{equation}

Let $l'=a_{1}a_{2}\ldots a_{|l'|}$, where each $a_{i}\in\{a,a^{-1}\}$, and let $l=w_{1}a_{1}w_{2}\ldots a_{|l'|}w_{|l'|+1}$ be the decomposition where each step $a_{i}$ is at $t$-height 0. We say that the $a$-height of one of the subwords $w_{i}$ is equal to the integer $p$ which satisfies $a^{p}=a_{1}\ldots a_{i}$. Then $l$ is a loop if and only if for any height $h$, concatenating all of the words $w_{i}$ at $a$-height $h$ creates a loop. Note that each word $w_{i}$ must have height 0 and have no $a$ or $a^{-1}$ steps at height 0. As in the previous section we define another generating function $\Lambda_{d}(x,y)$ by
\[\Lambda_{d}(x,y)=\sum H_{d_{n}}\prod_{j=1}^{n-1}x^{2d_{j+1}}y^{2d_{j+1}}{d_{j}+d_{j+1}-1\choose d_{j}-1}H_{d_{j}+d_{j+1}}(x),\]
where the sum is over all sequences $n,d_{1},d_{2},\ldots,d_{n}$ with $d_{1}=d$. In the same way as in the previous section we get the following equations, which are essentially the same as \eqref{FinO} and \eqref{Om}.

\begin{align}\label{GinO}G(x,y)=&\sum_{c,d=1}^{\infty}x^{2c+2d}y^{2c+2d}{c+d\choose c}H_{c+d+1}(x,y)\Lambda_{d}(x,y)\Lambda_{c}(x,y)\nonumber\\
+&2\sum_{d=1}^{\infty}x^{2d}y^{2d}H_{d}(x,y)\Lambda_{d}(x,y)\nonumber\\
+&H_{1}(x,y).\end{align}

\begin{equation}\label{Om2}\Lambda_{d}(x)=H_{d}(x,y)+\sum_{b=1}^{\infty}x^{2b}y^{2b}{b+d-1\choose d-1}H_{b+d}(x,y)\Lambda_{b}(x).\end{equation}

So now to calculate $G(x,y)$, we just need to calculate the generating functions $H_{n}(x,y)$. For each $k\in\mathbb{Z}_{\geq0}$, let $J_{k}(x)$ be the generating function for loops in $B$ which have exactly $k$ steps which end at $t$-height 0, none of which are $a$ or $a^{-1}$ steps, and which never go below height 0. For each such word $w$, the number of ways of breaking it into $n$ words $w_{1},w_{2},\ldots,w_{n}$ where each ends at height $0$, such that $w_{1}\ldots w_{n}=w$ is equal to
\[{k+n-1\choose n-1}.\]
Therefore, we can calculate each generating function $H_{n}(x,y)$ in terms of the generating functions $J_{k}(x)$ as follows:
\begin{equation}\label{Heq}H_{n}(x,y)=\sum_{k=0}^{\infty}y^{k}{k+n-1\choose n-1}J_{k}(x).\end{equation}

Finally, we will calculate the generating functions $J_{k}(x)$. Trivially we have $J_{0}(x)=1$. For $k>0$, let $l$ be a loop counted by $J_{k}(x)$. Then $l$ must contain exactly $k$ steps which end at height 0, which are not $a$ or $a^{-1}$ steps. Hence they must all be $t^{-1}$ steps. Therefore, $l$ decomposes as
\[l=tu_{1}t^{-1}tu_{2}t^{-1}\ldots t u_{k}t^{-1},\]
where each word $u_{k}$ ends at height 0 and never goes below height 0. Moreover, since $l$ is a loop, we must have $\overline{u_{1}\ldots u_{k}}=1$. Hence the word $u=u_{1}\ldots u_{k}$ is counted by the generating function $G(x,y)$. Moreover, if $u$ contains $m$ steps which end at height 0, then there are exactly
\[{m+k-1\choose k-1}\]
ways to decompose $u$ into subwords $u_{1},\ldots, u_{k}$ which each end at height 0. Hence we get the equation
\begin{equation}\label{Jeq}J_{k}(x)=\sum_{m=0}^{\infty}x^{2k}{m+k-1\choose k-1}[y^{m}]G(x,y).\end{equation}

Now using equations \eqref{GinO}, \eqref{Om2}, \eqref{Heq} and \eqref{Jeq} as well as the base case $J_{0}(x)=1$, we can calculate the coefficients of $G(x,y)$ in polynomial time using a dynamic program. Finally we need to relate these coefficients to the total number of loops in $B$. We claim that for each $n$, the number of loops $b_{n}$ of length $n$ in $B$ over the generating set $\{a,t,a^{-1},t^{-1}\}$ is equal to
\[b_{n}=\sum_{m=0}^{\infty}\frac{n}{m}[y^{m}][x^{n}]G(x,y).\]
The reason for this is that the contribution to both sides of the equation from any set of $n$ loops which are cyclic permutations of each other is the same. That is, if we take $n$ loops $x_{i}\ldots x_{n}x_{1}\ldots x_{i-1}$ for $1\leq i\leq n$, and $m$ of these are counted by $G(x,y)$, then they will each contribute $x^{n}y^{m}$ to $G(x,y)$, so altogether these will contribute $n$ to both sides of the equation. If two or more of these loops are identical, then we must have $x_{1}\ldots x_{n}=(x_{1}\ldots x_{p})^q$ for some $p,q$ satisfying $pq=n$. In this case, assuming that $q$ is maximal, the contribution to each side is $n/q$ instead of $n$, since we overcounted by a factor of $q$.

Using the last equation we can quickly calculate the coefficients of the cogrowth generating function $C_{B}(x)$ using those of $G(x,y)$. In Table \ref{tab:nb50} we give the first 50 coefficients of this generating function. In fact we have 128 terms.

\scriptsize{
\begin{table}
   \centering
   \begin{tabular}{@{}c @{}} 
      \hline    
     1\\
                                                           4\\
                                                          28\\
                                                         232\\
                                                        2092\\
                                                       19864\\
                                                      195352\\
                                                     1970896\\
                                                    20275692\\
                                                   211825600\\
                                                  2240855128\\
                                                 23952786400\\
                                                258287602744\\
                                               2806152315048\\
                                              30686462795856\\
                                             337490492639512\\
                                            3730522624066540\\
                                           41422293291178872\\
                                          461802091590831904\\
                                         5167329622166765872\\
                                        58012358366319158872\\
                                       653272479274904359312\\
                                      7376993667962247094112\\
                                     83518163933592420945440\\
                                    947797532286760923097848\\
                                  10779770914124700529470264\\
                                 122856228305621394118000520\\
                                1402877847412263986004347872\\
                               16048147989560391552043686160\\
                              183892883412730524613883088808\\
                             2110556326150834244975990231512\\
                            24259510831181186885644198829344\\
                           279244563297679787781517160899820\\
                          3218641495385722409923501191862264\\
                         37146337262307758446419466115479416\\
                        429227600058421313330040967935014416\\
                       4965493663308539362541734301378311648\\
                      57506535582014868288482236767840209688\\
                     666700108804771886996957763509359246064\\
                    7737176908622194648339548498436658811432\\
                   89878279784970230837678375953110478795352\\
                 1045033044367535197025078407316665177933928\\
                12161645115366917947524997117208173413019632\\
               141653302005285175865456465524239660635389712\\
              1651274058730064356309776255817393993665780288\\
             19264448513399180870635082273788105896265150480\\
            224919270246185854430934219198103161122414157760\\
           2627954546552385827255336138747466100454012242528\\
          30726935577139566309665785537931570627782996384120\\
         359517978960007312327796870699755173605904761839752\\
                   \hline    
   \end{tabular}
   \vspace{4mm}
   \caption{The first 50 coefficients of the cogrowth series for the Navas-Brin group $B.$}
   \label{tab:nb50}
\end{table}}
\normalsize

\subsection{A General Algorithm}
Before we describe the algorithm which we use for Thompson's group $F$, we will describe a general algorithm which can be applied to any group admitting certain functions which can be computed very quickly. In the next subsection we will describe how we apply this algorithm to $F$. This algorithm could also be applied to any of the other groups which we have discussed, however it would be much less efficient than the specific algorithms described previously in this section.
 
Our algorithm can be seen as a significantly more memory efficient version of the algorithm in \cite{EFR10}. First we describe that algorithm. Given a loop $\gamma=a_{0}a_{1}\ldots a_{2n}$, where each $a_{i}\in V(\Gamma)$ and $a_{2n}=a_{0}=e$, we define the midpoint of $\gamma$ to be the vertex $a_{n}$. Then $\gamma$ is made up of a walk of length $n$ from $e$ to its midpoint followed by a walk of length $n$ from its midpoint to 1. Hence, the number of loops in $\Gamma$ of length $2n$ with midpoint $m$ is the square of the number of walks of length $n$ from 1 to $m$.

Using a simple dynamic program, the algorithm calculates the number of walks to each vertex in $B(e,n)$, the ball of radius $n$ in $\Gamma$. Then one sums the squares of these numbers to calculate the number of loops of length $2n$. Note also that for each walk from $e$ to $m$, there is a corresponding walk from $e$ to $m^{-1}$, so it is only necessary to calculate the number of walks to either $m$ or $m^{-1}$. The problem with this algorithm is that it is necessary to store a large proportion of the ball of radius $n$ in memory at the same time. As a result it is essentially impossible to get any more than 24 coefficients of the cogrowth series for Thompson's group $F$ using this algorithm. Our algorithm is very similar except that we only store the ball of radius $k$ in memory, where $k\approx n/2$. Importantly, we do this without significantly increasing the running time of the program.

Let $G$ be a group with inverse closed generating set $S$. Let $\Gamma(G,S)$ denote the Cayley graph of $G$ with respect to the generating set $S$. We will often refer to this as simply $\Gamma$. We will assume that every loop has even length, however this algorithm could easily be altered to apply when this is not the case.

Let $O$ be an object in the program which represents an element of $G$. We require the following functions to be implemented:
\begin{itemize}
\item $init()$. This returns an object $O$ which represents the identity in $G$.
\item $val(O)$. This returns a value which is uniquely determined by the element of $G$ which the object $O$ represents. In other words, $val(O_1)=val(O_2)$ if and only if $O_1$ and $O_2$ represent the same element of $G$.
\item For each generator $\lambda\in S$, we have an operation $O.do_\lambda$. If $O$ initially represents the element $g\in G$, this changes $O$ to an object which represents $g\lambda$.
\item For each generator $\lambda\in S$, we have a function $l_{\lambda}(O)$, defined by $l_{\lambda}(O)=|g\lambda|-|g|$, where $g$ is the element of $G$ which $O$ represents. That is, $l_{\lambda}(O)=1$ if applying $\lambda$ moves $g$ away from the identity.
\end{itemize}
The speed of our algorithm depends entirely on the efficiency of these functions. For Thompson's group our implementations of these all take constant time. Importantly, we do not require an inverse of $val$ to be implemented.

Given these functions, the algorithm proceeds as follows:

\textbf{Step 1:} Assign an arbitrary order to the generating set $S$ and set $k=\lceil\frac{n}{2}\rceil$.

\textbf{Step 2:} Using a simple dynamic program, construct an associative array $A_{n-k}$, implemented as a hash table, with a key value pair $(k_{g},a_{g})$ for each element $g\in G$ within the ball of radius $n-k$. The key $k_{g}$ is given by $val(O)$ where $O$ is any object which represents $g$ and the value $a_{g}$ is equal to the number of walks of length $n-k$ in $\Gamma$ from $e$ to $g$. We will write $a_{g}=A[k_{g}]$. For a number $x$ which is not a key in $A_{n-k}$, we set $A[x]=0$.

\textbf{Step 3:} Construct a tree $T_{k}$ which contains one vertex $v_{g}$ for each element $g$ of $G$ within the ball of radius $k$, such that each vertex $v_{g}$, apart from $v_{e}$, is connected to exactly one vertex $v_{h}$ satisfying $|h|=|g|-1$, and $g=h\lambda$ for some $\lambda\in S$. If there are multiple possible choices of $h$, we choose the element $h$ which minimises $\lambda$, according to the order we assigned in step 1. The edge $(h,g)$ is then labelled with $\lambda$. Each vertex $v_{g}$ is also labelled with the number $p(v_{g})$ of paths of length $k$ in $\Gamma$ from $e$ to $g$.

\textbf{Step 4:} We now create a function $numpaths(O,d)$ whose input is an object $O$ and a positive integer $d$, which, assuming that $d=|g|$, outputs the number of paths of length $n$ in $\Gamma$ from $e$ to $g$, where $g$ is the group element represented by $O$. During the calculation of $ numpaths$ the object $O$ may change, but at the end it must represent the same group element $g$. Each path of length $n$ from $e$ to $g^{-1}$ in $\Gamma$ can be written in a unique way as a path of length $k$ from $e$ to some vertex $h$ in $\Gamma$ followed by a path of length $n-k$ from $h$ to $g^{-1}$. For a given $h$, the number of these paths is equal to $p(v_{h})A[k_{h^{-1}g^{-1}}]=p(v_{h})A[k_{gh}]$. Hence, the number which we need to return is
\[\sum_{h\in G}p(v_{h})A[k_{gh}].\]
Note also that the summand is 0 unless $|h|\leq k$ and $|gh|\leq n-k$, so we only need to sum over values of $h$ which satisfy these two inequalities. To do this we perform a depth first search of the tree $T_{k}$, skipping any sections where we can be sure that there are no vertices $v_{h}$ such that $h$ satisfies the two inequalities. We start the search at the root vertex $v_{e}$ of $T_{k}$ and initialise $r=0$ and $total=0$. Whenever we move from a vertex $v_{h}$ to $v_{h\lambda}$ we change $d$ to $d+l_{\lambda}(O)$ and then apply the operation $O.do_{\lambda}$. That way whenever we are at a vertex $v_{h}$, the object $O$ represents $gh$ and $d=|gh|$. We also increase $x$ by 1 whenever we move to a child vertex and decrease $x$ by 1 when we backtrack so that we always have $x=|h|$. Then we add $p(v_{h})A[k_{gh}]=p(v_{h})A[val(O)]$ to the sum $total$ if and only if $d\leq n-k$, since $x=|h|\leq k$ for every vertex $v_{h}$ in $T_{k}$. Since $d$ decreases by at most 1 when we move to a child vertex, and $x$ always increases by 1, the value $x+d$ never decreases when we move to a child vertex. So if $x+d>n$ when we are at a vertex $v_{h}$, then we do not traverse the children of $v_{h}$. At the end of the search we return to the root vertex so that $O$ is back to its original value and then return the value $total$.

\textbf{Step 5:} For the last step we just need to add up the value of $numpaths$ for every vertex $g$ in the ball of radius $n$ such that $|g|$ has the same parity as $n$. To accomplish this we perform a depth first search of the tree $T_{n}$, which is defined in the same way as $T_{k}$. However, we do not explicitly construct $T_{n}$ as doing so would use too much memory. In order to perform the depth first search, we just need a function $isedge_{\lambda}(O)$ for each $\lambda\in S$ which returns 1 if and only if there is an outward edge from $v_{g}$ to $v_{g\lambda}$ in $T_{n}$, where $g$ is the group element that $O$ represents. This will be the case if and only if $|g\lambda|=|g|+1$ and $|g\lambda\mu|=|g\lambda|+1$ for each $\mu\in S$ with $\mu<\lambda^{-1}$. We test this using the functions $l_{\lambda}$, $do_{\lambda}$ and $l_{\mu}$. During the depth first search, we keep track of the distance $d=|g|$, where $g$ is the group element represented by $O.$ Now, to calculate the number $numloops$ of loops of length $2n$, we first set $numloops=0$, then run the depth first search, and when we visit each vertex of $T_{n}$, add $numpaths(O,d)^{2}$ to $numloops$. At the end of this process $numloops$ is equal to the number of loops of length $2n$, so we return $numloops$ and terminate the algorithm.

The advantage of this algorithm is that it only stores $T_{k}$ and $A_{n-k}$ in memory, rather than all of $T_{n}$. This also allows us to parallelise step 5.

\subsection{Thompson's group $F$}
In this section we describe how the object $O$, the operation $do_{\lambda}$ and the functions $val$ and $l_{\lambda}$ are implemented for Thompson's group $F$. We use the standard generating set $S=\{a,b,a^{-1},b^{-1}\}$, which yields the presentation
\[F=\langle a,b|a^{2}ba^{-2}=baba^{-1}b^{-1},a^{3}ba^{-3}=ba^{2}ba^{-2}b^{-1}\rangle.\]
For $O$ we use the forest representation given by Belk and Brown in \cite{BB05}. We simultaneously store the forest diagram as a graph $P$ as well as a pair of binary strings $a,b$. A forest diagram is defined as a pair of sequences of binary trees, with one tree highlighted in each sequence. A single binary tree with $m$ leaves corresponds to a unique binary string $s$ of length $2m-2$ with the property that $s$ has an equal number of $1$'s and $0$'s and the number of $1$'s in any initial substring is at least equal to the number of $0$'s in that substring. This is defined by doing a depth first search of the tree and writing a 1 whenever we move down an edge from a vertex to its left subtree and writing a 0 whenever we backtrack along such an edge. Now to convert a sequence of binary trees to a binary string, we first convert each individual tree to a binary string, insert the string 01 before each such string, then concatenate the results. We then change the 01 before the string corresponding to the highlighted tree to 00. This is how the strings $a$ and $b$ are defined. We also store the numbers $p_{a}$ and $p_{b}$ in $O$, which define the positions of the 00 before the highlighted tree in each of $a$ and $b$. The strings $a$ and $b$ each have length at most $2n$, so they can be represented as 64 bit numbers as long as $n\leq32$. The operation $do_{\lambda}$ is defined easily for the effect on the graph $P$. The effect on the binary strings $a$ and $b$ is a bit more complicated and requires some bit shifting. The entire length of an element of Thompson's group $F$ can be determined by its forest diagram, as shown in \cite{BB05}, so we could use this to determine $l_{\lambda}$ by using the graph $P$ and simply subtracting the calculated length $|g|$ from the length we calculate for $|g\lambda|$. In fact we do it more efficiently than this, as the difference $|g\lambda|-|g|$ is determined entirely by the highlighted tree and the surrounding trees. Finally, $val(O)$ simply returns the pair $(a,b)$.

In Table \ref{tab:thomp} we give the first 32 coefficients of the cogrowth generating function for Thompson's group $F.$ This is 7 further terms than given in \cite{HHR15}.

\begin{table}[!ht]
   \centering
   \begin{tabular}{@{}c @{}} 
      \hline    
    Coefficients\\
\hline
1\\
4\\
28\\
232\\
2092\\
19884\\
196096\\
 1988452\\
 20612364\\
 217561120\\
2331456068\\
25311956784\\
277937245744\\
3082543843552\\
 34493827011868\\
 389093033592912\\
 4420986174041164\\
 50566377945667804\\
581894842848487960\\
6733830314028209908\\
78331435477025276852\\
 915607264080561034564\\
 10750847942401254987096\\
 126768974481834814357308\\
 1500753741925909645997904\\
 17833339046478612301547884\\
 212663448005862463186139032\\
 2544535423071442709522261116\\
30542557512715560857221200908\\
367718694478039302564802454628\\
4439941127401928226610731571976\\
53756708216952135677787623701460\\
      \hline    
   \end{tabular}
   \vspace{4mm}
   \caption{Terms in the cogrowth sequence of Thompson's group $F$.}
   \label{tab:thomp}
\end{table}

\section{Possible cogrowth of Thompson's Group}\label{sec:thom}
In this section we will show that if $a_{0},a_{1},\ldots$ is the cogrowth sequence for Thompson's group $F$, then for any real numbers $a<1$ and $\lambda>1$, the inequality
\[a_n<16^n\lambda^{-n^a}\]
holds for all sufficiently large integers $n$. As a result, if Thompson's group is amenable, then the sequence cannot grow at the rate \[16^n\lambda^{-n^a},\]
For any fixed $a<1$. This result follows quite readily from results in \cite{PS-C02} and \cite{PS-C00}, however we will need some definitions before we can see how they apply. Let $G$ be a group with finite generating set $S$. Then we define the function $\phi_{S}:\mathbb{Z}_{>0}\to\mathbb{R}_{>0}$ by setting $\phi_{S}(n)$ to be the probability that a random walk in $(G,S)$ of length $2n$ finishes at the origin. In other words, $|S|^{2n}\phi_{S}(n)$ is the number of loops of length $2n$ in the Cayley graph $\Gamma(G,S)$.

Now, for two different (non-increasing) functions $\phi_{1}$ and $\phi_{2}$, we say that $\phi_{1}\preceq\phi_{2}$, if there is some $C\in\mathbb{R}_{>0}$ such that
$\phi_{1}(n)\leq C\phi_{2}(n/C)$,
where each $\phi_{i}$ is extended to the reals by linear interpolation. Finally we say that $\phi_{1}\approx\phi_{2}$ if both $\phi_{1}\preceq\phi_{2}$ and $\phi_{2}\preceq\phi_{1}$. We recall Theorem 3.1 from \cite{PS-C00}:
\begin{thm} Let $G$ be a group with finite, symmetric generating set $S$ and let $H$ be a subgroup of $G$ and let $T$ be a finite symmetric generating set of $H$. Then
\[\phi_{S}\preceq\phi_{T}.\]\end{thm}
The other result we need concerns wreath products with $\mathbb{Z}$. In \cite{PS-C02}, Pittet and Saloff-Coste show (in a remark just below Theorem 8.11) that for a finite generating set $T$ of $\mathbb{Z}\wr_{d}\mathbb{Z}$, we have
\[\phi_{T}(n)\approx\exp\left(-n^{d/(d+2)}(\log n)^{2/(d+2)}\right).\]
 Now, since $\mathbb{Z}\wr_{d}\mathbb{Z}$ is a subgroup of Thompson's group $F$, we must have
\[\phi_{S}(n)\preceq\phi_{T}(n)\approx \exp\left(-n^{d/(d+2)}(\log n)^{2/(d+2)}\right),\]
where $S$ is the standard generating set of $F$. Hence, for any positive integer $d$, there is a positive real number $C$ such that
\[\phi_{S}(n)\leq C\exp\left(-(n/C)^{d/(d+2)}(\log (n/C))^{2/(d+2)}\right).\]
  
Now we are ready to prove our theorem. 
\begin{thm}\label{thm:3.2}
Let $a_{n}$ be the number of loops of length $2n$ in the standard Cayley graph for Thompson's group. Then for any real numbers $a<1$ and $\lambda>1$, the inequality
\[a_n<16^n\lambda^{-n^a}\]
holds for all sufficiently large integers $n$.\end{thm}
\begin{proof}Let $d$ be a positive integer such that $\frac{d}{d+2}>a$. Then there is some $C\in\mathbb{R}_{>0}$ such that
\[\phi_{S}(n)\leq C\exp\left(-(n/C)^{d/(d+2)}\log (n/C)^{2/(d+2)}\right)\]
for all $n\in\mathbb{Z}_{>0}$. For $n$ sufficiently large, we have $\log (n/C)>0$, so
\begin{align*}C\exp\left(-(n/C)^{d/(d+2)}\log (n/C)^{2/(d+2)}\right)<&C\exp\left(-(n/C)^{d/(d+2)}\right)\\
=&\exp\left(\log(C)-C^{-d/(d+2)}n^{d/(d+2)}\right),\end{align*}
Hence, for all $n$ sufficiently large we have
\[\phi_{S}(n)<\exp(-n^{\alpha}).\]
Therefore,
\[a_{n}=16^{n}\phi_{S}(n)<16^{n}\exp(-n^{\alpha})\]\end{proof}

Note that the same result holds if we replace Thompson's group $F$ with the Navas-Brin group $B$, since it also contains every wreath product $\mathbb{Z}\wr_{d}\mathbb{Z}$ as a subgroup.

\section{Moments}\label{sec:moments}

In \cite{HHR15}, Haagerup, Haagerup and Ramirez-Solano prove that the cogrowth sequence $a_{0},a_{1},\ldots$ for Thomson's group $F$ is the sequence of moments of some probability measure $\mu$ on $[0,\infty)$, in other words, the sequence is a Stieltjes moment sequence. In fact, their proof applies to the cogrowth series of any (locally finite) Cayley graph $\Gamma$. In this section, we generalise the result further, to any locally finite graph. First we  give some background on the Stieltjes and Hamburger moment problems.

\subsection{Stieltjes and Hamburger moment sequences}

In the following, for the sequence $\textit{\textbf{a}}=a_{0},a_{1},\ldots$, and $n\geq0$, we define the matrix $H_{\infty}^{(n)}(\textit{\textbf{a}})$ by
 \[H_{\infty}^{(n)}(\textit{\textbf{a}})=\begin{bmatrix}
    a_{n} & a_{n+1} & a_{n+2} & \dots \\
    a_{n+1} & a_{n+2} & a_{n+3} & \dots \\
    a_{n+2} & a_{n+3} & a_{n+4} & \dots \\
    \vdots & \vdots & \vdots & \ddots
\end{bmatrix}\]

\begin{thm}(Stieltjes \cite{S1894}, Gantmakher–Krein \cite{GK37} ) For a sequence $\textbf{a}=a_{0},a_{1},\ldots$, the following are equivalent:
\begin{itemize}
\item There exists a positive measure $\mu$ on $[0,\infty)$ such that 
\[a_{n}=\int x^{n}d\mu(x).\]
\item The matrices $H_{\infty}^{(0)}(\textbf{a})$ and $H_{\infty}^{(1)}(\textbf{a})$ are both positive semidefinite.
\item There exists a sequence of real numbers $\alpha_{0},\alpha_{1},\ldots\geq0$ such that the gnerating function $A(t)$ for the sequence $a_{0},a_{1},\ldots$ satisfies
\[A(t)=\sum_{n=0}^{\infty}a_{n}t^{n}=\frac{\alpha_{0}}{\displaystyle1-\frac{\alpha_{1}t}{\displaystyle1-\frac{\alpha_{2}t}{1-\ldots}}}\]

\end{itemize}
\end{thm}

A sequence which satisfies the conditions of the theorem above is called a Stieltjes moment sequence.

\begin{thm} For a sequence $\textbf{a}=a_{0},a_{1},\ldots$, the following are equivalent:
\begin{itemize}
\item There exists a positive measure $\mu$ on $(-\infty,\infty)$ such that 
\[a_{n}=\int x^{n}d\mu(x).\]
\item The matrix $H_{\infty}^{(0)}(\textbf{a})$ is positive semidefinite.

\end{itemize}
\end{thm}

A sequence which satisfies the conditions of the theorem above is called a Hamburger moment sequence. From either definition of Hamburger moment sequence, it follows immediately that any Stieltjes moment sequence is a Hamburger moment sequence. Carleman's condition states that the measure $\mu$ is unique if
\[\sum_{n=0}^{\infty}a_{2n}^{-\frac{1}{2n}}=+\infty,\]
This is certainly true when the sequence grows at most exponentially, as is the case for all of our examples. For Stieltjes moment sequences, the following weaker condition implies that the measure $\mu$ is unique:
\[\sum_{n=0}^{\infty}a_{n}^{-\frac{1}{2n}}=+\infty.\]
 For a Hamburger moment sequence $\textit{\textbf{a}}$, which grows at most exponentially, the radius of convergence of $A(t)=a_{0}+a_{1}t+a_{2}t^{2}+\ldots$ is equal to
\[\frac{1}{|\mu|}.\]
In particular, this means that if $\textit{\textbf{a}}$ is a Stieltjes moment sequence, the exponential growth rate of the sequence is equal to the minimum value in the support of $\mu$. 

One benefit of proving that  a sequence $\textit{\textbf{a}}$ is a Stieltjes moment sequence is that it allows us to compute good lower bounds for the exponential growth rate of the sequence using only finitely many terms. This method was described in \cite{HHR15}, but we repeat the description here, using the continued fraction form of $\textit{\textbf{a}}$. We consider the generating function
\[A(t)=\sum_{n=0}^{\infty}a_{n}t^{n}=\frac{\alpha_{0}}{\displaystyle1-\frac{\alpha_{1}t}{\displaystyle1-\frac{\alpha_{2}t}{1-\ldots}}}.\]
Using the terms $a_{0},\ldots,a_{n}$, we calculate the terms $\alpha_{0},\ldots,\alpha_{n}$. It is easy to see that $A(t)$ is nondecreasing in each $\alpha_{j}$. Hence, the minimum possible value $A_{n}(t)$ is achieved by setting $\alpha_{n+1},\alpha_{n+2},\ldots$ to 0. Therefore, the radius of convergence $t_{c}$ of $A(t)$ is bounded above by the radius of convergence $t_{n}$ of $A_{n}(t)$. Therefore, $b_{n}=1/t_{n}$ is a lower bound for the exponential growth rate of $\textit{\textbf{a}}$. It is easy to check that the sequence $b_{1},b_{2},\ldots$ is nondecreasing, and $b_{n}^{n}>a_{n}/a_{0}$. It follows that this sequence of lower bounds converges to the exponential growth rate of $\textit{\textbf{a}}$.

If we assume further that the sequences $\alpha_{0},\alpha_{2},\alpha_{4},\ldots$ and $\alpha_{1},\alpha_{3},\ldots$ are non-decreasing, as seems to be true for many of the cases we consider, we can get stronger lower bounds for the growth rate by setting $\alpha_{n+1},\alpha_{n+3},\ldots$ to $\alpha_{n-1}$ and $\alpha_{n+2},\alpha_{n+4},\ldots$ to $\alpha_{n}$. For this sequence the exponential growth rate of corresponding sequence $\textit{\textbf{a}}$ is $(\sqrt{\alpha_{n}}+\sqrt{\alpha_{n-1}})^{2}$.

\subsection{Applications of moments to the cogrowth series}

Here we  describe how to compute lower bounds for the growth rate of the cogrowth sequence of Thompson's group $F.$

\begin{thm}Let $\Gamma$ be a locally finite graph with a fixed base vertex $v_{0}$. For each $n\in\mathbb{Z}_{\geq0}$, let $t_{n}$ be the number of loops of length $n$ in $\Gamma$ which start and end at $v_{0}$. Then there exists a probability measure $\mu$ on $\mathbb{R}$ such that for each $n\in\mathbb{Z}_{\geq0}$, the $n$th moment of $\mu$ is given by
\[\int_{-\infty}^{\infty}x^{n}d\mu=t_{n}.\]
In other words, $t_{0},t_{1},\ldots$ is a Hamburger moment sequence.
\end{thm}
\begin{proof} The sequence $\textit{\textbf{t}}=t_{0},t_{1},\ldots$ is a Hamburger moment sequence if and only if the matrix \[H_{\infty}^{(0)}(\textit{\textbf{t}})=\begin{bmatrix}
    t_{0} & t_{1} & t_{2} & \dots \\
    t_{1} & t_{2} & t_{3} & \dots \\
    t_{2} & t_{3} & t_{4} & \dots \\
    \vdots & \vdots & \vdots & \ddots
\end{bmatrix}\]
is positive semidefinite. To prove this, we will show that this is the matrix representation of a positive definite bilinear form.

Let $M$ be the inner product space over $\mathbb{R}$ defined by the orthonormal basis $\{b_{v}|v\in V(\Gamma)\}$. For each $n\in\mathbb{Z}$ we let $x_{n}\in M$ be the element defined by 
\[x_{n}=\sum_{v\in V(\Gamma)}p_{v}b_{v},\]
where $p_{v}$ is the number of paths of length $n$ in $\Gamma$ from $v_{0}$ to $v$. Then it is easy to see that for any non-negative integers $m$ and $n$, the value $\langle x_{n},x_{m}\rangle$ is equal to the number $t_{m+n}$ of paths of length $m+n$ in $\Gamma$ from $v_{0}$ to itself. Now let $X$ be the subspace of $M$ spanned by $\{x_{0},x_{1},\ldots\}$. Then $A$ is the matrix representation of the inner product $\langle,\rangle$, restricted to $X$, with respect to the spanning set $\{x_{0},x_{1},\ldots\}$. Therefore, $H_{\infty}^{(0)}(\textit{\textbf{t}})$ is positive semidefinite. Note that if $\{x_{0},x_{1},\ldots\}$ are linearly independent, then $H_{\infty}^{(0)}(\textit{\textbf{t}})$ is positive definite.
\end{proof}

\begin{thm}Let $C\in\mathbb{Z}_{>0}$ and let $\Gamma$ be a graph with a fixed base vertex $v_{0}$, such that each vertex in $\Gamma$ has degree at most $C$.  For each $n\in\mathbb{Z}_{\geq0}$, let $t_{n}$ be the number of loops of length $n$ in $\Gamma$ which start and end at $v_{0}$. There exists a probability measure $\mu$ on $\mathbb{R}_{>0}$ such that for each $n\in\mathbb{Z}_{\geq0}$, the $n$th moment of $\mu$ is equal to $t_{2n}$. In other words, $t_{0},t_{2},t_{4},\ldots$ is a Stieltjes moment sequence.

Moreover, $\mu$ is unique and its support is contained in the interval $[0,C^{2}]$.  
\end{thm}

\begin{proof} In order to show that the sequence $\textit{\textbf{s}}=t_{0},t_{2},t_{4},\ldots$ is a Stieltjes moment sequence, it suffices to prove that the two matrices
\[H_{\infty}^{(0)}(\textit{\textbf{s}})=\begin{bmatrix}
    t_{0} & t_{2} & \dots \\
    t_{2} & t_{4} & \dots \\
    \vdots & \vdots & \ddots 
\end{bmatrix}
~~~~~\text{and}~~~~~
H_{\infty}^{(1)}(\textit{\textbf{s}})=\begin{bmatrix}
    t_{2} & t_{4} & \dots \\
    t_{4} & t_{6} & \dots \\
    \vdots & \vdots & \ddots
\end{bmatrix}\]
are positive semidefinite. From the previous theorem, we know that the matrix
\[H_{\infty}^{(0)}(\textit{\textbf{t}})=\begin{bmatrix}
    t_{0} & t_{1} & t_{2} & \dots \\
    t_{1} & t_{2} & t_{3} & \dots \\
    t_{2} & t_{3} & t_{4} & \dots \\
    \vdots & \vdots & \vdots & \ddots
\end{bmatrix}\]
is positive semidefinite. Hence any principal submatrix of $H_{\infty}^{(0)}(\textit{\textbf{t}})$ (using the same rows and columns) is also positive semidefinite. Since both the matrices $H_{\infty}^{(0)}(\textit{\textbf{s}})$ and $H_{\infty}^{(1)}(\textit{\textbf{s}})$ are such principal submatrices of $H_{\infty}^{(0)}(\textit{\textbf{t}})$, each of these matrices is positive semidefinite. Therefore, the sequence $t_{0},t_{2},t_{4},\ldots$ is a Stieltjes moment sequence. Now, since each vertex of the graph has degree at most $C$, the number of paths of length $n$ is at most $C^{n}$. Hence we have the inequality
\[\int_{0}^{\infty}x^{n}d\mu=t_{2n}\leq C^{2n}.\]
Therefore, the support of $\mu$ must be contained in the interval $[0,C^{2}]$. This also implies that $\mu$ is unique.  
\end{proof}

In particular, if we let $G$ be a finitely generated group, with inverse closed generating set $S$, and $\Gamma$ be the corresponding Cayley graph, then the even terms of the cogrowth sequence for $\Gamma$ form a Stieltjes moment sequence. Moreover, each vertex has degree $|S|$ so the support of the corresponding measure $\mu$ is contained in the interval $[0,|S|]$. As described in the previous subsection, we can compute lower bounds $b_{n}$ for the exponential growth rate of any such sequence. Turning our attention to Thompson's group, using 31 terms of the cogrowth sequence, we have computed the corresponding terms $\alpha_{0},\alpha_{1},\ldots,\alpha_{31}$. Using these we have computed the rigorous lower bound $b_{31}\approx13.269$ for the exponential growth rate of the cogrowth sequence of Thompson's group. If we assume that the sequences $\alpha_{0},\alpha_{2},\ldots$ and $\alpha_{1},\alpha_{3},\ldots$ are increasing, we get the stronger lower bound $(\sqrt{\alpha_{30}}+\sqrt{\alpha_{31}})^2\approx13.706$. In Section \ref{sec:Thompson} below we extrapolate the sequence of bounds $\{b_n\}$ to estimate the growth constant $\mu,$ and find $\mu \approx 15.0.$
 
\section{Series Analysis}\label{sec:analysis}
We have series for six groups, which we will consider in order. Firstly, the group ${\mathbb Z}^2,$ then the Heisenberg group, the lamplighter group  $L=\mathbb{Z}_{2}\wr\mathbb{Z},$ the two groups $\mathbb{Z}\wr\mathbb{Z}$ and $( \mathbb{Z}\wr\mathbb{Z})\wr\mathbb{Z},$ the Navas-Brin group $B$ \cite{N04,B05} and finally Thompson's group $F$. We will analyse each of these in turn. 

In all cases our initial analysis is based on the behaviour of the ratio of successive terms, with other methods deployed as appropriate. In the simplest situation we consider, which is when the asymptotic form of the coefficients is $c_n \sim c \cdot \mu^n \cdot n^g,$ one has that the {\it ratio} of successive coefficients is asymptotically linear when plotted against $1/n,$ as
\begin{equation} \label{ratios}
r_n = \frac{c_n}{c_{n-1}}=\mu \left (1 + \frac{g}{n} + {\rm o}\left (\frac{1}{n}\right )\right ).
\end{equation}
It is therefore natural to plot the ratios $r_n$ against $1/n.$
If the correction term ${\rm o}\left (\frac{1}{n}\right )$ can be ignored\footnote{In the simplest cases, such as the present one, the correction term will be ${\rm O}\left (\frac{1}{n^2}\right ).$}, such a plot will be linear,
with gradient $\mu \cdot g,$ and intercept $\mu$ at $1/n = 0.$ If the growth constant $\mu$ is known, or can be guessed, better estimates of the exponent $g$ can be made by extrapolating the sequence $$g_n = (r_n/\mu-1)\cdot n = g + o(1).$$

More complicated asymptotic forms for the coefficients can give rise to different expressions for the ratios, as we show below.

\subsection{The group ${\mathbb Z}^2$.}
For the group ${\mathbb Z}^2,$ the coefficients of the cogrowth series are known exactly, $c_n = {2n \choose n}^2,$ and so the ratio of successive terms is $$r_n=\frac{c_n}{c_{n-1}}=16 \left ( 1 - \frac{1}{n} + \frac{1}{4n^2} \right ).$$

A plot of the ratios against $1/n$ is shown in Figure \ref{fig:zsq-rat}, based on the first 50 coefficients. It is clearly going to the expected limit of 16. The exponent $g$ should be $-1,$ and we plot estimators $g_n$ against $1/n$ in Figure \ref{fig:zsq-g}, which is also clearly going to the expected limit $-1.$ This corresponds to a logarithmic singularity of the generating function, $$C_{\mathbb{Z}^2}(x) \sim c\cdot \log(1 - 16x).$$

\begin{figure}
 \setlength{\captionindent}{0pt}
\begin{minipage}{0.48\textwidth}
\includegraphics[width=0.97\linewidth]{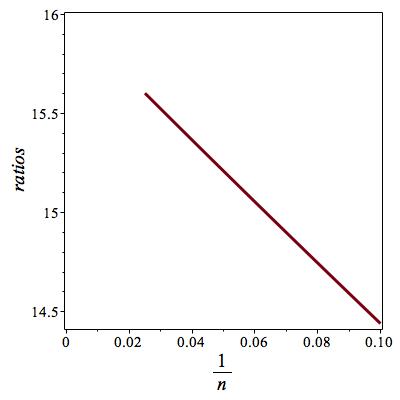}
\captionof{figure}{ Plot of ${\mathbb Z}^2$ ratios against $1/n.$}
\label{fig:zsq-rat}
\end{minipage}\hfill
\begin{minipage}{0.48\textwidth}
\includegraphics[width=0.97\linewidth]{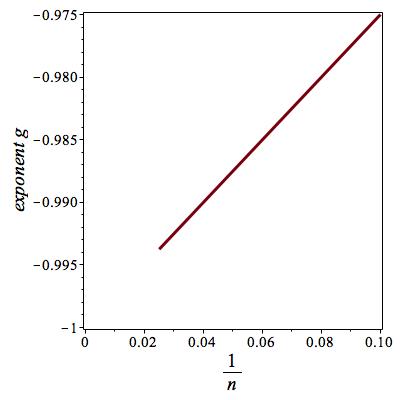}
\captionof{figure}{Estimators of exponent $g$ for ${\mathbb Z}^2$ vs. $1/n.$}
\label{fig:zsq-g}
\end{minipage}
\end{figure}

For this simple example one can do much better by using the package {\em gfun}, available in Maple, and asking for the underlying ordinary differential equation for the generating function, given the first 20 or so coefficients. In this way one immediately obtains the result for the generating function $$C_{\mathbb{Z}^2}(x) = \sum c_n x^n =  2 {\bf K} \left ( \frac{4\sqrt{x}}{\pi} \right ),$$ where ${\bf K} $ is the complete elliptic integral of the first kind. 

\subsection{The Heisenberg group.}

We have calculated 90 terms of the generating function, and show that this is sufficient to obtain a very precise asymptotic representation of the coefficients.
The leading order asymptotics of the coefficients is known \cite{G11} to be $c_n \sim {16^n }/(2n^2),$ corresponding to a generating function $$C_{Heisenberg} \sim \frac{1}{2} (1-16x)\log(1-16x).$$ 
We have analysed this series in the same way as described above for the group ${\mathbb Z}^2$.

A plot of the ratios against $1/n$ is shown in Figure \ref{fig:heis-rat}. It is clearly going to the expected limit of 16. The exponent $g$ should be $-2,$ and we plot estimators $g_n$ against $1/n$ in Figure \ref{fig:heis-g}, which are also clearly going to the expected limit $-2.$

\begin{figure}
\setlength{\captionindent}{0pt}
\begin{minipage}{0.48\textwidth}
  \includegraphics[width=0.97\linewidth]{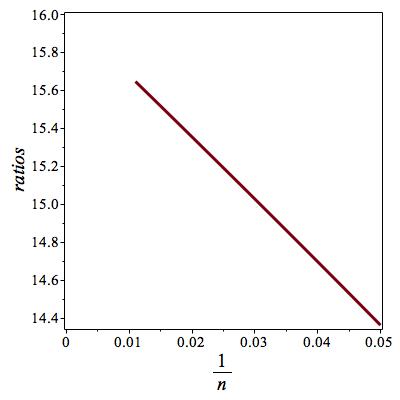}  
   \caption{Plot of Heisenberg group ratios against $1/n.$}
   \label{fig:heis-rat}
\end{minipage}\hfill
\begin{minipage}{0.48\textwidth}
 \includegraphics[width=0.97\linewidth]{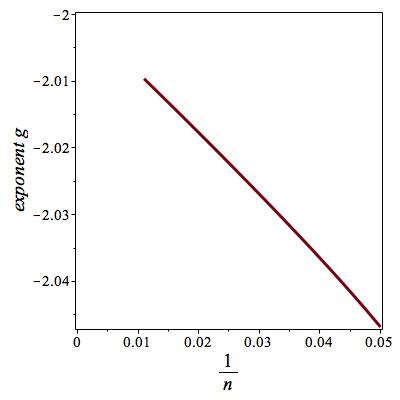}  
  \caption{Estimators of exponent $g$ for the Heisenberg group vs. $1/n.$}
   \label{fig:heis-g}
\end{minipage}
\end{figure}

In order to obtain higher-order asymptotic terms, we subtract the known leading-order term from the coefficients, forming the sequence $$c^{(1)}(n) = c_n- {16^n }/(2n^2).$$ A ratio analysis of this sequence strongly suggests that $c^{(1)}(n) \sim const/n,$ implying that $c_n \sim {16^n }/(2n^2)+ const./n^3.$ Such behaviour is consistent with a simple algebraic singularity of the generating function. Accordingly, we attempted a linear fit to the assumed form $c_n/16^n = 1/(2n^2)+ k_1/n^3+ k_2/n^4 + k_3/n^5.$ We did this by solving the linear system given by setting $n=m-1, \,\, n=m, \,\, n=m+1$ in the preceding equation, and solving for $k_1,\,\, k_2, \,\, k_3,$
with $m$ ranging from 20 to the maximum possible value 89. We obtain an $m$-dependent sequence of estimates of the amplitudes $k_1,\,\, k_2, \,\, k_3,$ which we extrapolated against appropriate powers of $1/m.$ 

In this way we estimate $k_1=0.93341,$ $k_2=1.530,$ and $k_3=3.30,$ where we expect errors in these estimates to be confined to the last quoted digit. 

To summarise, we find the asymptotics of the coefficients of the cogrowth series of the Heisenberg group to be $$c_n = 16^n \left( \frac{1}{2n^2}+ \frac{0.93341}{n^3}+ \frac{1.530}{n^4} + \frac{3.30}{n^5} + O\left (\frac{1}{n^6} \right ) \right ).$$

\subsection{The lamplighter group.}\label{lamp}
The lamplighter group $L$ is the wreath product of the group of order two with the integers, $L=\mathbb{Z}_{2}\wr\mathbb{Z}.$  From \cite{R03} we know that for this group, 
\begin{equation} \label{coeff}
c_n \sim c \cdot 9^n \cdot \kappa^{n^{1/3}} \cdot n^{1/6}.
\end{equation}
 So in this example we see the presence of a stretched-exponential term, $\kappa^{n^{1/3}},$ which makes the analysis more difficult. As remarked above, we have generated 201 terms of the cogrowth series, and show how these terms can be used to estimate the asymptotic behaviour of the coefficients.

If the coefficients of a series include a stretched-exponential term, so that $$a_n \sim c \cdot \mu^n \cdot \kappa^{n^\sigma} \cdot n^g,$$ with $0 < \sigma, \, \kappa < 1,$ then the ratio of successive terms behaves as 
$$r_n = \frac{a_n}{a_{n-1}} \sim \mu \left ( 1 + \frac{\sigma\log {\kappa}}{n^{1-\sigma}} + \frac{g}{n} + \cdots \right ).$$ Experimentally, the presence of such a stretched-exponential term is signalled by the fact that the ratio plots against $1/n$ exhibit curvature, and that this curvature can be eliminated, or at least substantially reduced, by plotting the ratios against $1/n^{1-\sigma},$ where $\sigma$ is roughly estimated by choosing its value so as to maximise linearity. This theory is developed in greater detail, along with several examples, in \cite{G15}. 

Because of the presence of two terms in the ratio plots, one of order O$(n^{\sigma-1})$ the other of order O$(1/n),$ there is some competition between these two terms, which can make it difficult to estimate the value of $\sigma$ just from the linearity of the ratio plots. So we first eliminate the O$(1/n)$ term by calculating the modified ratios 
\begin{equation} \label{mod-rat}
r_n^{(1)} = n\cdot r_n - (n-1) \cdot r_{n-1} = \mu \left ( 1 + \frac{\sigma^2\log {\kappa}}{n^{1-\sigma}} + o\left (\frac{1}{n}\right )  \right ).
\end{equation}
 In Figure \ref{fig:l1} we show the modified ratios plotted against $1/n^{2/3},$ which is seen to be linear, and extrapolating to the known growth constant of  9. While not shown, we also plotted the modified ratios against $1/\sqrt{n}$ and against $1/n^{3/4}.$ These were visibly convex upward and concave downward, respectively. One would conclude that $1/2 < \sigma < 3/4,$ and bearing in mind that in all known such behaviour, $\sigma$ is a simple rational fraction (arguably simply related to dimensionality), one would conjecture that $\kappa=2/3.$ However, we can also estimate the value of $\sigma$ by other means.

\begin{figure}[htbp] 
 \centering
   \includegraphics[width=3in]{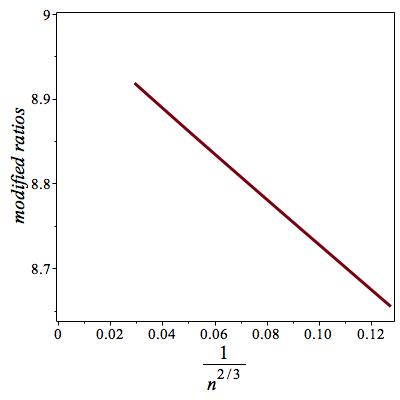} 
   \caption{ Modified lamplighter group ratios vs. $n^{-2/3}.$}
   \label{fig:l1}
\end{figure}

If we assume $\mu=9,$ then from (\ref{mod-rat}) it follows that a plot of $l_n = \log|1- r_n^{(1)}/\mu|$ against $\log(n)$ should be linear with gradient $\sigma-1.$ This plot (not shown) is indeed visually linear. To calculate the gradient, which will vary slightly with $n,$ we calculate the local gradient $(l_n-l_{n-1})/(\log(n)-\log(n-1)),$ and show this plotted against $1/n^{4/3}$ in Figure \ref{fig:l2}. This plot is clearly going to a limit very close to $-2/3,$ as expected.

\begin{figure}[h!]
\setlength{\captionindent}{0pt}
\begin{minipage}{0.48\textwidth}
\centerline { \includegraphics[width=0.97\linewidth]{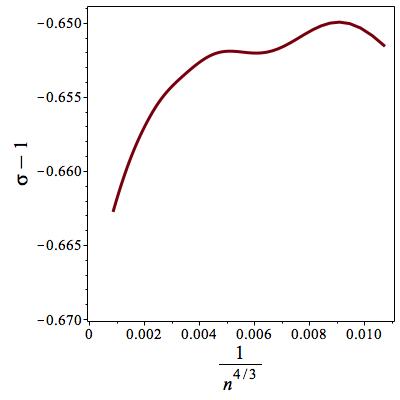}   }
   \caption{ Estimates of $\sigma-1$ vs. $n^{-4/3}.$}
   \label{fig:l2}
\end{minipage}\hfill
\begin{minipage}{0.48\textwidth}
\centerline {  \includegraphics[width=0.97\linewidth]{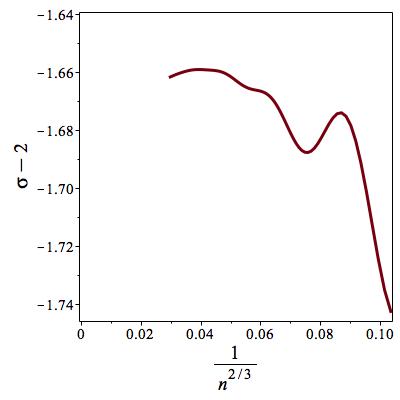}   }
  \caption{ Estimates of $\sigma-2$ vs. $n^{-2/3}.$}
   \label{fig:l3}
\end{minipage}
\end{figure}

One can also find estimators for the exponent $\sigma$ without assuming or knowing the value of the growth constant $\mu.$ Taking the ratio of the modified ratios eliminates the growth constant $\mu,$ so that $$r_n^{(2)} = \frac{r_n^{(1)}}{r_{n-1}^{(1)}} = 1 - \frac{\sigma^2(1-\sigma)\log{\kappa}}{n^{2-\sigma}}+o(n^{\sigma-2}).$$ So a plot of $ \log| r_n^{(2)}-1|$ against $\log{n}$ should be linear with gradient $\sigma-2.$ As above, we don't show this uninteresting linear plot, but instead show the local gradient, plotted against $1/n^{2/3},$ in Figure \ref{fig:l3}, which appears to be going to a value around $-1.67,$ consistent with the known exact value $-5/3.$

Assuming the values $\mu=9,$ and $\sigma=1/3,$ we estimate the remaining parameters in the asymptotic expression by direct fitting to the logarithm of the coefficients. From $c_n \sim c \cdot 9^n \cdot \kappa^{n^{1/3}} \cdot n^{g}$ we get $$\log{c_n}- n\cdot \log{9} \sim n^{1/3} \cdot \log{\kappa} + g \cdot \log{n} + \log{c}.$$  As in the preceding analysis of the Heisenberg group coefficients, we fit successive triples of coefficients to get estimates of the three unknowns, $\log{\kappa},$ $g$ and $\log{c}.$ The results are shown in Figures \ref{fig:l4}, \ref{fig:l5}, and \ref{fig:l6} respectively.

\begin{figure}[h!]
\setlength{\captionindent}{0pt}
\begin{minipage}{0.34\textwidth}
\centerline {  \includegraphics[width=0.9\linewidth]{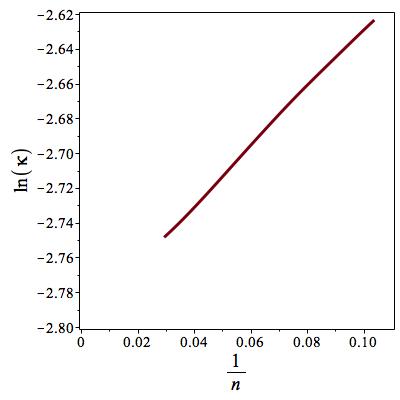}    }
   \caption{ Estimates of $\log{\kappa}$ vs. $1/n.$}
   \label{fig:l4}
\end{minipage}\hfill
\begin{minipage}{0.3\textwidth}
\centerline {   \includegraphics[width=0.97\linewidth]{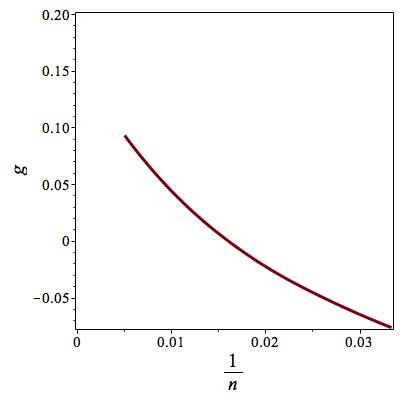}   }
  \caption{ Estimates of exponent $g$ vs. $1/n.$ The exact value is $1/6.$}
   \label{fig:l5}
\end{minipage}\hfill
\begin{minipage}{0.33\textwidth}
\centerline {  \includegraphics[width=0.97\linewidth]{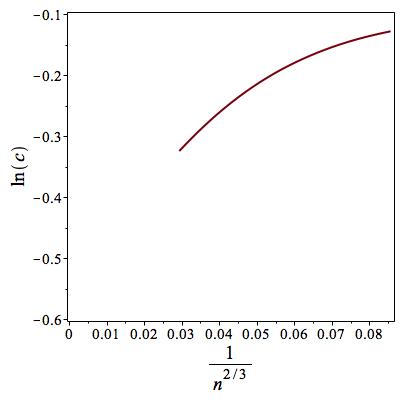}  }
  \caption{ Estimates of $\log{c}$ vs. $n^{-2/3}.$}
   \label{fig:l6}
\end{minipage}
\end{figure}

From these plots, we estimate $\log {\kappa} \approx -2.78,$ $ g \approx 0.17,$ and $\log {c} \approx -0.6.$ If we use the fact that we know that the exponent $g=1/6,$ we can get refined estimates of the remaining parameters, giving $\log {\kappa} \approx -2.775,$  and $\log {c} \approx -0.55,$ so that $\kappa \approx 0.0623,$ and $c \approx 0.58.$ As far as we are aware, these two constants have not previously been estimated.

\subsection{Analysis of group $\mathbb{Z} \wr {\mathbb Z}$}\label{zsquared}

As discussed in the introduction, for the groups $\mathbb{Z} \wr_d {\mathbb Z},$ there is an  additional logarithmic factor associated with the stretched-exponential term. For $d=1$
the group $\mathbb{Z} \wr {\mathbb Z}$ has coefficients that behave as $$a_n \sim const \cdot \mu^n \cdot \kappa^{n^\sigma \log^\delta{n}} \cdot n^g, \,\, {\rm with} \,\,  \sigma= 1/3 \,\, {\rm and}\,\, \delta =2/3.$$ 
It follows that the ratio of successive coefficients behaves as
\begin{equation}\label{ratio1}
r_n = \frac{a_n}{a_{n-1}} \sim \mu \left ( 1 + \frac{\sigma \cdot \log{\kappa}\cdot \log^{\delta}{n}}{n^{1-\sigma}} +  \frac{\delta \cdot \log{\kappa}\cdot \log^{\delta-1}{n}}{n^{1-\sigma}} +\frac{g}{n} + \cdots \right ).
\end{equation}
 
We have generated series to order $x^{276}$ for this group. A simple ratio plot against $1/n$ is strongly concave downwards. Plotting the ratios against $1/n^{2/3}$ gives a plot  which is much closer to linear, but still displays a slight concavity. A simple ratio plot against $1/\sqrt{n}$ by contrast, displays slight convexity.




As we noted in our analysis of the lamplighter group, the term $g/n$ in eqn. (\ref{ratio1}) also makes a contribution (as does the logarithmic term $\log^\delta{n}$), so a clearer picture emerges if this term is eliminated, which we do  by forming the modified ratios (\ref{mod-rat}),  which behave in this case as
\begin{equation} \label{case1}
r_n^{(1)}=  \mu \left ( 1 + \frac{\log{\kappa}}{9n^{2/3}}\left (\log^{2/3}{n}+4\log^{-1/3}{n}-2\log^{-4/3}{n} \right ) +o(n^{-5/3+\epsilon}) \right ).
\end{equation}

 Plots of the modified ratios are shown in Figures \ref{fig:l7},  \ref{fig:l8}, and  \ref{fig:l9}, against $1/\sqrt{n},$ $1/n^{2/3}$ and $1/n^{3/4}$ respectively. It is clear that the plot against $1/n^{2/3}$ is the closest to linear, corresponding to $\kappa = 1/3.$ However, there is still some downward concavity, due to the associated logarithmic terms. To see this even more clearly, we show in Figure \ref{fig:l10} a plot of the modified ratios against $\left (\log^{2/3}{n}+4\log^{-1/3}{n}-2\log^{-4/3}{n} \right )/n^{2/3},$ which is the expected asymptotic behaviour, see (\ref{case1}). This is indistinguishable from linearity.
 
 \begin{figure}[h!]
\setlength{\captionindent}{0pt}
\begin{minipage}{0.48\textwidth}
\centerline { \includegraphics[width=0.97\linewidth]{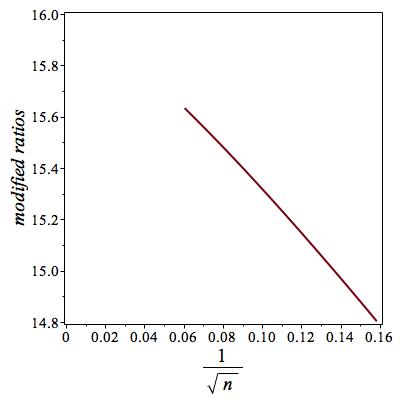}    }
    \caption{Modified ratios for $ \mathbb{Z} \wr {\mathbb Z}$ vs. $1/\sqrt{n}.$}
   \label{fig:l7}
\end{minipage}\hfill
\begin{minipage}{0.48\textwidth}
\centerline { \includegraphics[width=0.97\linewidth]{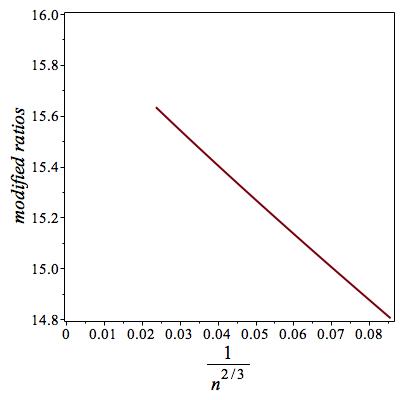} }
   \caption{Modified ratios for $ \mathbb{Z} \wr {\mathbb Z}$ vs. $n^{-2/3}.$}
   \label{fig:l8}
\end{minipage}
\end{figure}

\begin{figure}[h!]
\setlength{\captionindent}{0pt}
\begin{minipage}{0.48\textwidth}
\centerline { \includegraphics[width=0.97\linewidth]{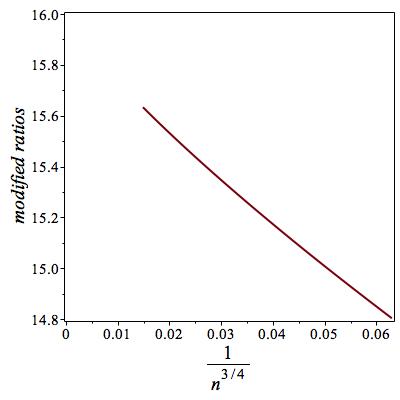} }
   \caption{Modified ratios for $ \mathbb{Z} \wr {\mathbb Z}$ vs. $n^{-3/4}.$}
   \label{fig:l9}
\end{minipage}\hfill
\begin{minipage}{0.48\textwidth}
\centerline {   \includegraphics[width=0.97\linewidth]{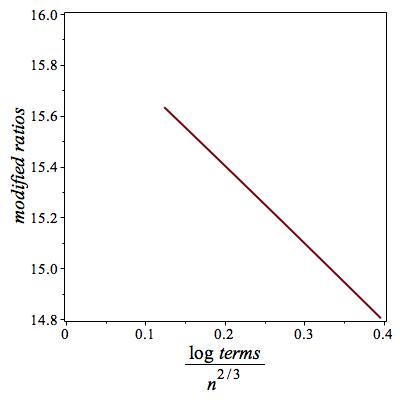} }
   \caption{Modified ratios for $ \mathbb{Z} \wr {\mathbb Z}$ vs.  $\left (\log^{2/3}{n}+4\log^{-1/3}{n}-2\log^{-4/3}{n} \right ) n^{-2/3}.$}
   \label{fig:l10}
\end{minipage}
\end{figure}

To date we haven't tried to estimate $\mu,$ known to be exactly 16. One way to do this is from the modified ratio plots shown above. All are seen to be tracking towards a value very close to 16.

It is also possible to estimate the exponent $\sigma$ directly from the ratios, even without knowing the dominant exponential growth constant $\mu.$ One first forms the ratio of successive ratios, so that
\begin{equation}\label{rratio1}
rr_n^{(1)} = \frac{r_n}{r_{n-1}} = 1 + \frac{\log{\kappa}\cdot \log^{\delta}{n}}{n^{2-\sigma}} \left ( \sigma(\sigma-1)+\frac{\delta(2\sigma-1)}{\log{n}}+\frac{\delta(\delta-1)}{\log^2{n}} \right ) -\frac{g}{n^2}+ o(1/n^2).\end{equation}

As we did above with the ratios, we eliminate the $O(1/n^2)$ term by constructing a modified ratio-of-ratios,
\begin{equation}\label{eqn:rr1}
rr_n^{(2)} = \frac{n^2 rr_n^{(1)} - (n-1)^2 rr_{n-1}^{(1)}}{2n-1} = 1 + \frac{ c\log{^\delta}{n}}{n^{2-\sigma}}\left (1 +O(1/\log{n})\right ),
\end{equation}
 where the constant $c=(\sigma^2(\sigma-1)\log{\kappa})/2.$

Then a plot of $\log|rr_n^{(2)}-1|$ against $\log{n}$ should be close to linear, as the logarithmic term will vary very slowly over the range of $n$-values at our disposal, with gradient $\sigma-2.$ Such a plot (not shown) {\em is} visually linear, but in order to calculate the gradient we find the (local) gradient of the segment joining
$rr_n^{(2)}$ and $rr_{n-1}^{(2)},$ which should approach the ``correct'' value as $n$ increases. This is shown, plotted against $1/n$ in Figure \ref{fig:l11}. It appears to be going to a limit around $-1.62\,\, {\rm to} \,\,-1.61,$ which would imply $\sigma \approx 0.38\,\, {\rm or}\,\, 0.39,$ rather than the known value of $1/3.$

\begin{figure}[htbp] 
 \centering
   \includegraphics[width=3in]{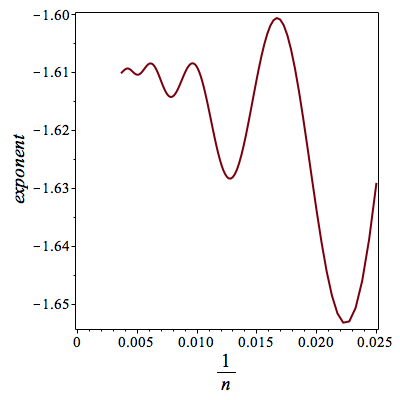} 
   \caption{Estimators of exponent $\sigma-2$ vs. $1/n.$}
   \label{fig:l11}
\end{figure}

 However, if we assume we know that $\delta=2/3,$ and include the confluent logarithmic term $\log^{2/3}{n}$ in the exponent of the stretched-exponential term, plotting instead 
 $$\log\left ( \frac{r_n^{(2)}-1}{\log^{2/3}{n}} \right )$$ 
 against $\log{n},$ the plot is again visually linear. However the corresponding plot of the local gradient, shown in Figure \ref{fig:l12}, is clearly going to a limit around $-5/3,$ consistent with the known value $\sigma=1/3.$

\begin{figure}[htbp] 
\setlength{\captionindent}{0pt}
\begin{minipage}{0.48\textwidth}
   \includegraphics[width=0.97\linewidth]{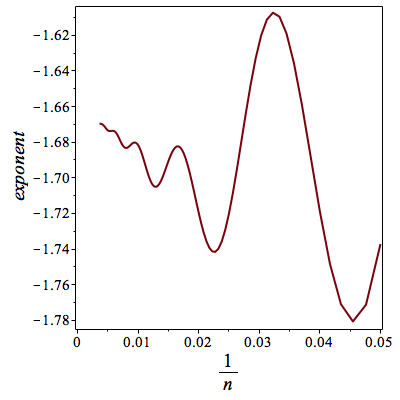} 
   \caption{Estimators of exponent $\sigma-2$ vs. $1/n,$ assuming a confluent logarithmic term.}
   \label{fig:l12}
\end{minipage}\hfill
\begin{minipage}{0.48\textwidth} 
   \includegraphics[width=0.97\linewidth]{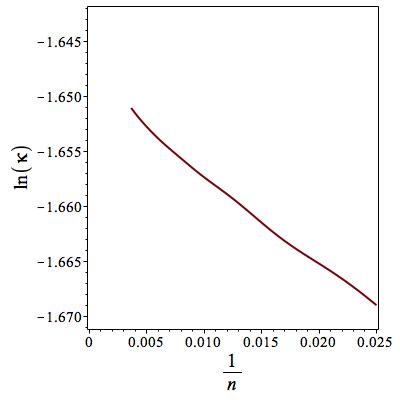} 
   \caption{ Estimates of $\log{\kappa}$ vs. $1/n.$}
   \label{fig:l13}
     \end{minipage}
\end{figure}
 
 Assuming the values $\mu=16,$ and $\sigma=1/3$  and $\kappa=2/3,$ we can estimate the remaining parameters in the asymptotic expression by direct fitting to the logarithm of the coefficients. From $c_n \sim c \cdot 16^n \cdot \kappa^{n^{1/3}\log^{2/3}n} \cdot n^{g}$ we get $$\log{c_n}- n\cdot \log{16} \sim n^{1/3}\cdot \log^{2/3}{n} \cdot \log{\kappa} + g \cdot \log{n} + \log{c}.$$  As in the preceding analysis of the lamplighter group coefficients, we fit successive triples of coefficients to get $n-$dependant estimates of the three unknowns, $\log{\kappa},$ $g$ and $\log{c}.$ The results are shown in Figures \ref{fig:l13}, \ref{fig:l14}, and \ref{fig:l15} respectively.

\begin{figure}[htbp] 
\setlength{\captionindent}{0pt}
\begin{minipage}{0.48\textwidth}
   \includegraphics[width=0.97\linewidth]{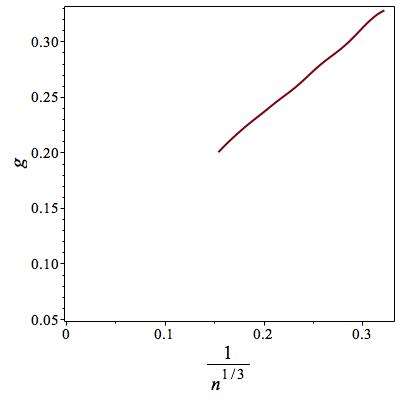} 
   \caption{Estimates of exponent $g$ vs. $n^{-1/3}.$ }
   \label{fig:l14}
\end{minipage}\hfill
\begin{minipage}{0.48\textwidth} 
   \includegraphics[width=0.97\linewidth]{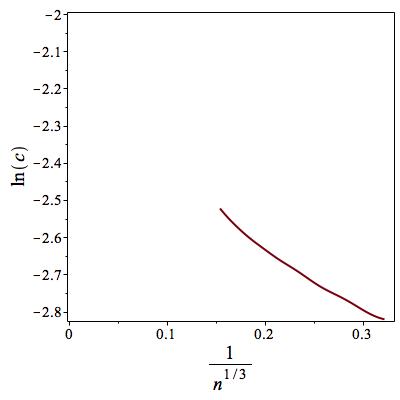} 
   \caption{ Estimates of $\log{c}$ vs. $n^{-1/3}.$}
   \label{fig:l15}
     \end{minipage}
\end{figure}

From these plots, we estimate $\log{\kappa} \approx -1.64,$ but it is difficult to estimate $g.$ It appears to be quite small, close to zero, and could even be negative. It is even more difficult to extrapolate the plot for $\log{c},$  though one might conclude the bound $\log{c} \ge -2.$ These estimates correspond to $\kappa \approx 0.194,$ $ g \approx 0,$ and $c > 0.13.$
As far as we are aware, these three constants have not previously been studied.

In anticipation of our analysis of Thompson's group $F,$ where the growth constant $\mu$ is not known, we attempt to estimate both the exponents $\sigma$ and $\delta$ without knowing the value of $\mu.$ Forming the ratios (\ref{ratio1}) eliminates the constant $c$ in the asymptotic form of the coefficients, and the ratio of ratios (\ref{rratio1}) eliminates $\mu.$ If we now form the sequence 
\begin{equation}\label{eqn:tn}
t_n = \frac{rr_n^{(1)}-1}{rr_{n-1}^{(1)}-1}
\end{equation}
 this eliminates the base $\kappa$ of the stretched-exponential term, and in fact $$n(t_n-1) \sim \sigma-2 +\frac{\delta}{n\log{n}}.$$ So plotting $n(t_n-1)$ against $1/(n\log{n})$ should give an estimate of $\sigma-2.$ To estimate $\delta,$ we form the sequence $$n\log^2{n}(n(t_n-1)-(n-1)(t_{n-1}-1)) \sim -\delta +O(1/\log{n}).$$ We show these plots in Figures \ref{fig:z2sig} and \ref{fig:z2del} respectively. The estimate of $\sigma-2$ appears to be going to a limit of around -1.6 or below, c.f. the known exact value of $-5/3,$ while the estimate of $\delta$ is harder to estimate, but the plot is certainly consistent with the known value $2/3.$ As can be seen, this exponent is difficult to estimate without many more terms than we currently have.

\begin{figure}[htbp] 
\setlength{\captionindent}{0pt}
\begin{minipage}{0.48\textwidth}
   \includegraphics[width=0.97\linewidth]{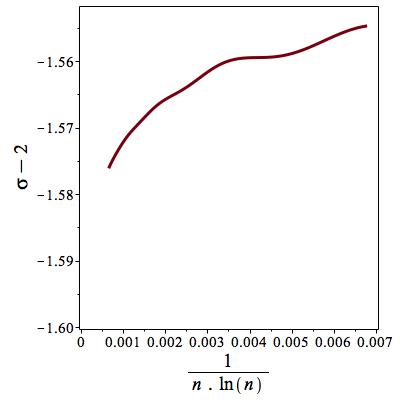} 
   \caption{ Estimates of $\sigma-2$ vs. $1/(n\log{n})$}
   \label{fig:z2sig}
\end{minipage}\hfill
\begin{minipage}{0.48\textwidth} 
 \centering
   \includegraphics[width=0.97\linewidth]{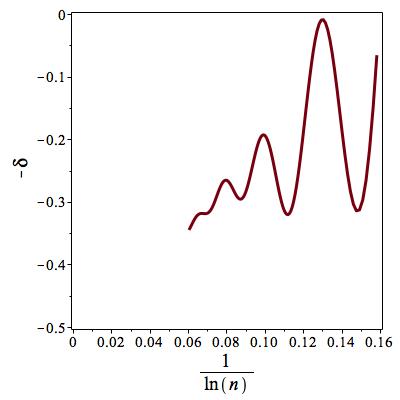} 
   \caption{ Estimates of exponent $-\delta$ vs. $1/\log{n}.$ }
   \label{fig:z2del}
   \end{minipage}
\end{figure}

 \subsection{Analysis of group $(\mathbb{Z} \wr {\mathbb Z})\wr {\mathbb Z}$}\label{zcubed}
 For this group we have 132 terms in the cogrowth series, just less than half the number we have for $ {\mathbb Z}\wr {\mathbb Z},$ so the results are not quite as precise.
We analysed this series the same way as for the group $\mathbb{Z} \wr {\mathbb Z}$.
 For this group it is known that the coefficients grow exponentially, and that the dominant term is $36^n.$ The sub-dominant term is $\kappa^{n^{1/2}\log^{1/2}{n}},$ which again follows from Theorem 3.11 in \cite{PS-C02}. Again, there is presumably a sub-sub dominant term $n^{g}.$

 In this  case we have for the ratio of successive terms:
 \begin{equation}\label{ratio2}
r_n = \frac{a_n}{a_{n-1}} \sim \mu \left ( 1 + \frac{\log{\kappa}\log^{1/2}{n}}{2n^{1/2}} +  \frac{\log{\kappa}}{2n^{1/2}\log^{1/2}{n}} +\frac{g}{n} + \cdots \right ).
\end{equation}

We eliminate the $O(1/n)$ term by forming the modified ratios (\ref{mod-rat}) which behave as

 \begin{equation}\label{case2a}
r_n^{(1)} =  \mu \left ( 1 + \frac{\log{\kappa}}{4\sqrt{n}}\left (\sqrt{\log{n}}+2\log^{-1/2}{n}-\log^{-3/2}{n} \right ) +o(n^{-3/2+\epsilon}) \right ).
\end{equation}

 First, we remark that extrapolating the ratios against $1/n$ gives a plot with considerable curvature (not shown). We plotted the modified ratios, defined above, against $1/n^{\sigma}$ for several values of $\sigma.$ We show the results for $\sigma=1/2$ and $\sigma=1/3$ in Figures \ref{fig:z31} and \ref{fig:z32} respectively. Surprisingly, the latter is closer to linear, however it extrapolates to a value of $\mu$ rather larger than the actual value, $\mu=36.$ However if we include the effect of the logarithmic term in the exponent, and plot (see equation (\ref{case2a})) the modified ratios against $\sqrt{ \frac{\log{n}}{n} },$ the modified ratio plot, shown in Figure \ref{fig:z33},  is indistinguishable from linearity and extrapolates to the correct value of $\mu.$

 \begin{figure}[h!]
\centering
   \includegraphics[width=3in]{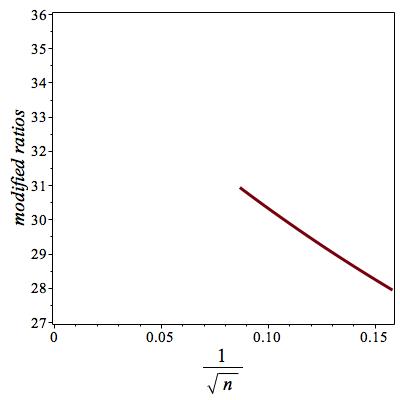} 
   \caption{Modified ratios for $ (\mathbb{Z} \wr {\mathbb Z})  \wr {\mathbb Z}$ vs. $1/\sqrt{n}.$}
   \label{fig:z31}
\end{figure}

\begin{figure}[htbp] 
\setlength{\captionindent}{0pt}
\begin{minipage}{0.48\textwidth}
   \includegraphics[width=0.97\linewidth]{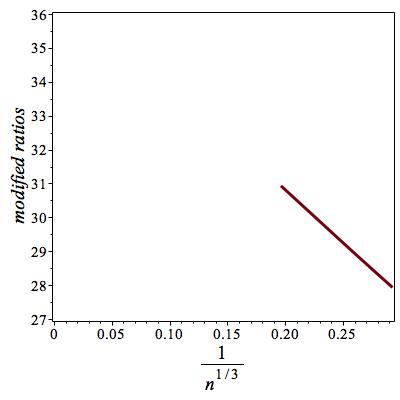} 
   \caption{Modified ratios for $( \mathbb{Z} \wr {\mathbb Z})  \wr {\mathbb Z}$ vs. $n^{-1/3}.$}
   \label{fig:z32}
\end{minipage}\hfill
\begin{minipage}{0.48\textwidth}
   \includegraphics[width=0.97\linewidth]{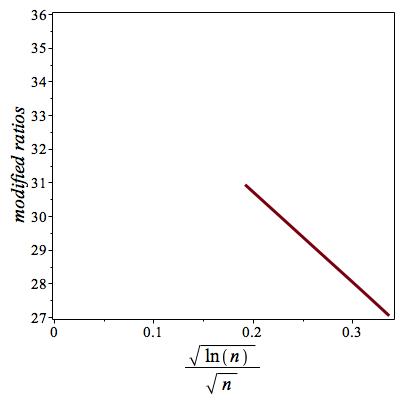} 
   \caption{Modified ratios for $( \mathbb{Z} \wr {\mathbb Z} ) \wr {\mathbb Z}$ vs.  $\sqrt{\log{n}/n}
.$}
   \label{fig:z33}
   \end{minipage}
\end{figure}

Repeating the analysis of the previous section, we attempted to estimate the exponent $\sigma$ without assuming the value of the growth constant $\mu.$
A plot of $\log|rr_n^{(2)}-1|$  (\ref{eqn:rr1}) against $\log{n}$ should be close to linear, (as the logarithmic term will vary only slowly over the range of $n$-values at our disposal), with gradient $\sigma-2.$ Such a plot (not shown) {\em is} visually linear, but in order to calculate the gradient we find the (local) gradient of the segment joining
$rr_n^{(2)}$ and $rr_{n-1}^{(2)},$ which should approach the ``correct'' value as $n$ increases. This is shown, plotted against $1/n$ in Figure \ref{fig:z3s1}. It appears to be going to a limit below $-1.42$
 which would imply $\sigma <  0.58,$ compared to the known value of $1/2.$

\begin{figure}[htbp] 
\setlength{\captionindent}{0pt}
 \begin{minipage}{0.48\textwidth} 
   \includegraphics[width=0.97\linewidth]{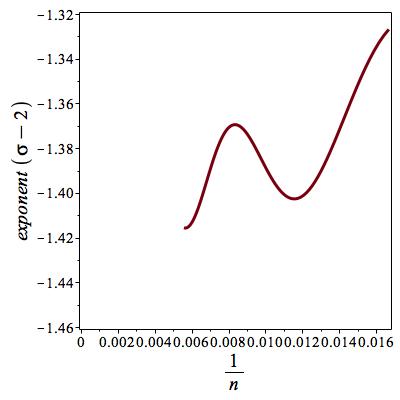} 
   \caption{Estimators of exponent $\sigma-2$ vs. $1/n.$}
   \label{fig:z3s1}
\end{minipage}\hfill
\begin{minipage}{0.48\textwidth} 
   \includegraphics[width=0.97\linewidth]{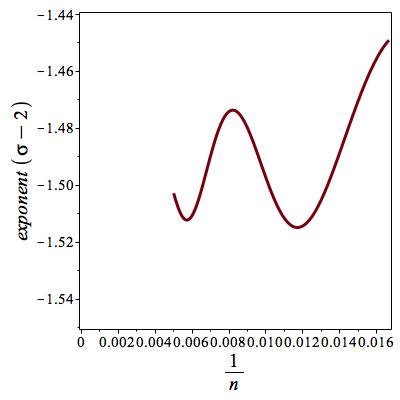} 
   \caption{Estimators of exponent $\sigma-2$ vs. $1/n,$ assuming a confluent logarithmic term.}
   \label{fig:z3s2}
   \end{minipage}
\end{figure}

 However, if we assume we know that $\delta=1/2,$ and include the confluent logarithmic term $\log^{1/2}{n}$ in the exponent of the stretched-exponential term, plotting instead 
 $$\log\left ( \frac{r_n^{(2)}-1}{\log^{1/2}{n}} \right )$$ 
 against $\log{n},$ the plot is again visually linear. Moreover the corresponding plot of the local gradient, shown in Figure \ref{fig:z3s2}, is going to a limit around $-3/2,$ consistent with the known value $\sigma=1/2.$

 Assuming the values $\mu=16,$  $\sigma=1/2$  and $\kappa=1/2,$ we can estimate the remaining parameters in the asymptotic expression by direct fitting to the logarithm of the coefficients. From $c_n \sim c \cdot 36^n \cdot \kappa^{n^{1/2}\log^{1/2}n} \cdot n^{g}$ we get $$\log{c_n}- n\cdot \log{36} \sim n^{1/2}\cdot \log^{1/2}{n} \cdot \log{\kappa} + g \cdot \log{n} + \log{c}.$$  As in the preceding analysis of ${\mathbb Z} \wr {\mathbb Z}$, we fit successive triples of coefficients to get estimates of the three unknowns, $\log{\kappa},$ $g$ and $\log{c}.$ The results for the first two are shown in Figures \ref{fig:z3s3} and \ref{fig:z3s4} respectively. From this, and further analysis with an additonal term in the assumed asymptotic form, we estimate $\log{\kappa} \approx -2.3$ and $g \approx 3.3.$

 \begin{figure}[h!]
\setlength{\captionindent}{0pt}
\begin{minipage}{0.48\textwidth}
   \includegraphics[width=0.97\linewidth]{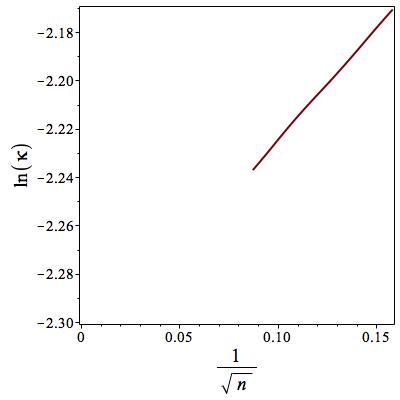} 
   \caption{ Estimates of $\log{\kappa}$ vs. $1/\sqrt{n}.$}
   \label{fig:z3s3}
\end{minipage}\hfill
\begin{minipage}{0.48\textwidth} 
   \includegraphics[width=0.97\linewidth]{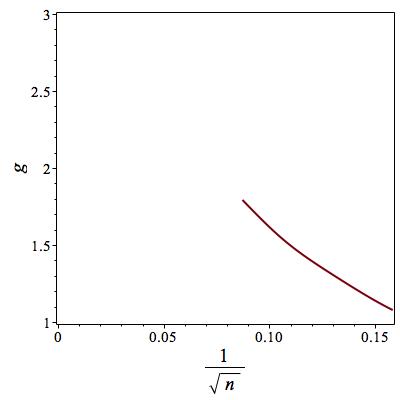} 
   \caption{ Estimates of exponent $g$ vs. $1/\sqrt{n}.$ }
   \label{fig:z3s4}
   \end{minipage}
\end{figure}

Again repeating the analysis of the previous section, we tried to estimate $\sigma$ and $\delta$ directly without knowing $\mu$ or $\kappa.$   Plotting $n(t_n-1)$ (\ref{eqn:tn})  against $1/(n\log{n})$ should give an estimate of $\sigma-2,$ and plotting the sequence $n\log^2{n}(n(t_n-1)-(n-1)(t_{n-1}-1))$ against $1/\log{n}$ should give estimates of $-\delta.$ We show these plots in Figures \ref{fig:z3sig} and \ref{fig:z3del} respectively. The estimate of $\sigma-2$ appears to be going to a limit of below -1.39 or so, c.f. the known exact value of $-1.5,$ while it is not possible to estimate $\delta$ from this plot, but it is not inconsistent with the known value $1/2.$

 \begin{figure}[h!]
\setlength{\captionindent}{0pt}
\begin{minipage}{0.48\textwidth}
   \includegraphics[width=0.97\linewidth]{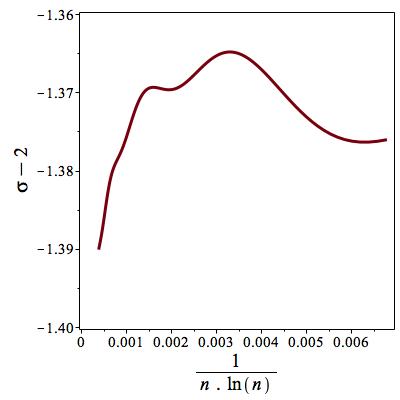} 
   \caption{ Estimates of $\sigma-2$ for $ (\mathbb{Z} \wr {\mathbb Z})  \wr {\mathbb Z}$ vs. $1/(n\log{n})$}
   \label{fig:z3sig}
\end{minipage}\hfill
\begin{minipage}{0.48\textwidth} 
 \centering
   \includegraphics[width=0.97\linewidth]{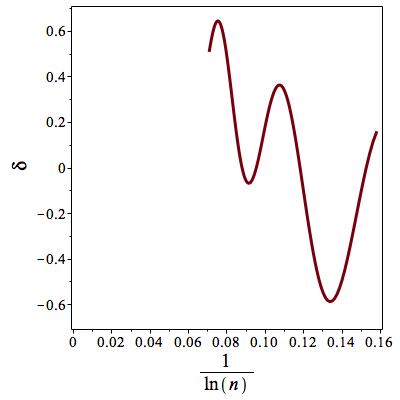} 
   \caption{ Estimates of exponent $\delta$ for $( \mathbb{Z} \wr {\mathbb Z})  \wr {\mathbb Z}$ vs. $1/\log{n}.$ }
   \label{fig:z3del}
\end{minipage}
\end{figure}

\subsection{ The group $\mathbb{Z} \wr_d  \mathbb{Z} $}\label{z98}
In the previous sections we have considered the analysis of the groups $\mathbb{Z} \wr_d  \mathbb{Z} $ for $d=1$ and $d=2.$ We have shown how the stretched-exponential term slows the rate of convergence of the ratios, but that appropriate analysis can still reveal much asymptotic information. However as $d$ increases, it becomes increasingly difficult to extract the asymptotics from a hundred or so terms of the cogrowth series. To see this, we consider the case $d=98.$ Then we know the asymptotic form of the coefficients is $$c_n \sim c \cdot \mu^n \cdot \kappa^{n^\sigma \log^\delta{n}} \cdot n^g,$$ where $\sigma = 49/50$ and $\delta=1/50$ \cite{PS-C02}.

While we could have generated 100 or so terms of this series from the algorithms described above, it will be more instructive to generate a test series with the given asymptotic behaviour, as then we can generate thousands of terms essentially immediately. 

So we have generated coefficients defined by $c_n = c \cdot \mu^n \cdot \kappa^{n^\sigma \log^\delta{n}} \cdot n^g$ with $c=1,$ $\mu=4,$ $\kappa=0.7,$ $g=0.5,$  
 $ \sigma = 49/50$ and $\delta=1/50.$ The ratio of successive terms must go to 4.0, the value of the growth constant\footnote{The growth constant is actually $4(d+1)^2,$ but for this exercise the actual value is irrelevant, so we have chosen a much smaller value.}. Using 128 terms of this test series, we show a plot of the ratios against $1/n$ in Figure \ref{fig:z98-rat1}. It is not possible to assert that, as $n \to \infty$ the ratios will go to 4.0. In Figure \ref{fig:z98-rat2} we show the same plot with 1280 terms. While this curve is steeply increasing, it is still not possible to assert that the limiting value is 4.0. Using 10000 terms, and plotting the ratios against $1/n^{1/50}$ (not shown), we finally see evidence that the extrapolated limit is around 3.8 or 3.9.

 \begin{figure}[h!]
\setlength{\captionindent}{0pt}
\begin{minipage}{0.48\textwidth}
\centerline { \includegraphics[width=0.97\linewidth]{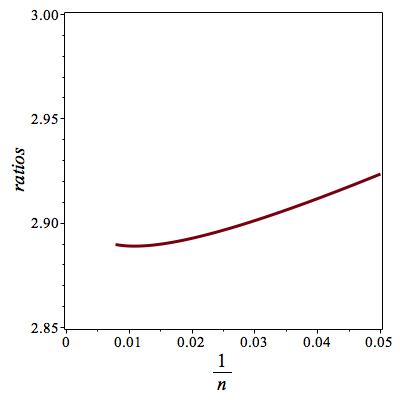} }
   \caption{The first 128 ratios for $\mathbb{Z} \wr_{98}  \mathbb{Z} $ vs. $1/n.$}
   \label{fig:z98-rat1}
\end{minipage}\hfill
\begin{minipage}{0.48\textwidth}
\centerline {   \includegraphics[width=0.97\linewidth]{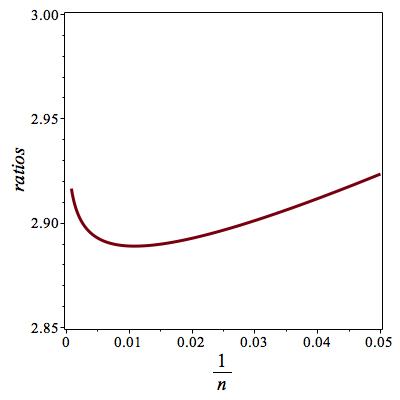} }
   \caption{The first 1280 ratios for $\mathbb{Z} \wr_{98}  \mathbb{Z} $ vs. $1/n.$}
   \label{fig:z98-rat2}
\end{minipage}
\end{figure}
 
 For this series the asymptotic form of the ratios is $$r_n = \mu \left (1 + \frac{49\log\kappa}{50\cdot n^{1/50}}+ \frac{0.5}{n} + o(1/n) \right ), $$ so we might expect more informative results if we eliminate the term $O(1/n),$ which we can do by forming the modified ratios. These are shown, plotted against $1/n^{1/50}$ in Figures \ref{z98modrat1} and \ref{z98modrat2}, based on  the first 128 terms and the first 10000 terms. Extrapolating these to $n \to \infty$ again gives a limit around 3.9.

 \begin{figure}[h!]
\setlength{\captionindent}{0pt}
\begin{minipage}{0.48\textwidth}
   \includegraphics[width=0.97\linewidth]{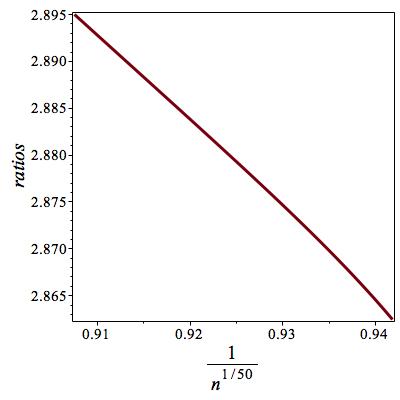} 
   \caption{The first 128 modified ratios for $\mathbb{Z} \wr_{98}  \mathbb{Z} $ vs. $n^{-1/50}.$}
   \label{z98modrat1}
\end{minipage}\hfill
\begin{minipage}{0.48\textwidth}
   \includegraphics[width=0.97\linewidth]{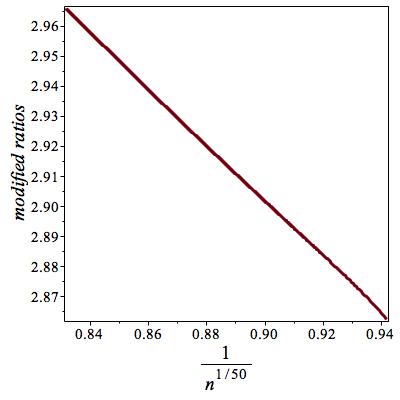} 
   \caption{ The first 10000 modified ratios for $\mathbb{Z} \wr_{98}  \mathbb{Z} $ vs. $n^{-1/50}.$ }
   \label{z98modrat2}
\end{minipage}
\end{figure}

It is possible to estimate the exponent $\sigma$ without knowing $\mu,$ as we showed in previous examples above. In particular, using the method based on equation (\ref{rratio1}), and described immediately below that equation, we show in Figure \ref{fig:sigestz98} a plot of estimators of $\sigma-2$ against $1/n,$ based on a 10000 term series, and it is persuasively going to the known value $-1.02.$

 \begin{figure}[h!]
\centering
   \includegraphics[width=3in]{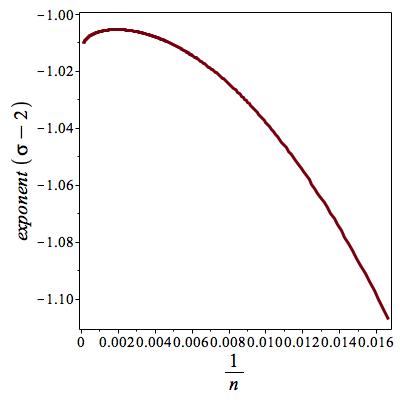} 
   \caption{Estimators of $\sigma-2$ for $ \mathbb{Z} \wr_{98} {\mathbb Z} $ against $1/n$ for $n \le 10000.$}
   \label{fig:sigestz98}
\end{figure}
 
Unfortunately, for no interesting problem is it realistic to get 10000 terms, so this example, and the next, must remain as a cautionary tale, to the extent that there can and do exist groups whose cogrowth series exhibit asymptotic behaviour that is difficult to estimate by numerical methods of the type we have considered.  Another example of similar difficulty is given by the Navas-Brin group $B,$ discussed in the next section.

\section {Series extension}\label{sec:extension}

In this section we develop one further tool that will be extremely useful in our analysis of the series for Thompson's group $F,$ where we have only 32 terms, rather than a hundred or more as in the examples we have been considering. It will also be very helpful in our analysis of the Navas-Brin group $B$, discussed in the next section.

Recall that our analysis of the more complex asymptotic forms that include stretched-exponential terms is based on ratios of successive terms, whereas for simpler groups, with simpler asymptotics, we used the method of differential approximants (DAs). It is obviously highly desirable to have further terms (in particular, further ratios), for all series with non-simple asymptotics, and particularly in those cases where we have comparatively short series, such as the 32 term series we have for Thompson's group $F.$ In order to obtain further ratios (or terms), we use the method of differential approximants {\em to predict subsequent ratios/terms}. The detailed description as to how this is done is given in \cite{G16}. 

We will give two demonstrations of the effectiveness of this method. In the first, we take the first 32 terms of the  series for ${\mathbb Z} \wr {\mathbb Z}$ discussed above, (we have more than 200 terms for this series), and use these to predict the next 89 ratios, from 5th order DAs. As well as the mean ratio, we calculate the standard deviation. We show, in Table \ref{tab:llerror}, a comparison between the actual error in the predicted ratios and the standard deviation of the estimated ratios. It can be seen that the true error lies between 1 and 1.5 standard deviations, which provides some confidence that the predicted ratios are accurate to within an error of 1.5 standard deviations.

For the series simulating the coefficients of the group $\mathbb{Z} \wr_{98}  \mathbb{Z}, $ we showed the importance of long series to reveal the asymptotic behaviour with some precision. In this second example, we take the first 100 terms of this series, and use them to predict the {\em next 315 ratios}. That is, we estimate $c_n/c_{n-1}$ for $n=101\cdots 415.$ 

To see how precisely these ratios can be predicted, we plot the difference between the actual ratios and those calculated by 4th order differential approximants in Figure \ref{fig:z98err}. It can be seen that the error is less than 2 parts in $10^{20}$ for all $n < 416.$ Just to make this perfectly clear, given 100 coefficients, we have predicted the next 315 ratios with an accuracy of some 20 significant digits.

 \begin{figure}[h!]
\setlength{\captionindent}{0pt}
\begin{minipage}{0.48\textwidth}
   \includegraphics[width=0.97\linewidth]{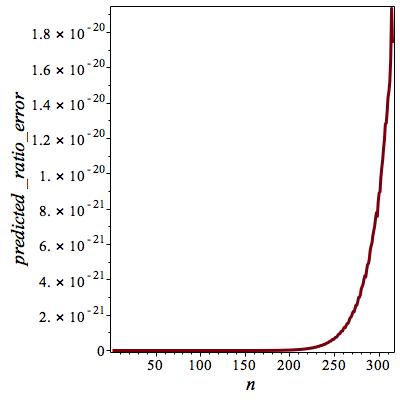} 
   \caption{Absolute error in predicted ratios of $\mathbb{Z} \wr_{98}  \mathbb{Z} $ for $100<n < 416.$}
   \label{fig:z98err}
\end{minipage}\hfill
\begin{minipage}{0.48\textwidth}
   \includegraphics[width=0.97\linewidth]{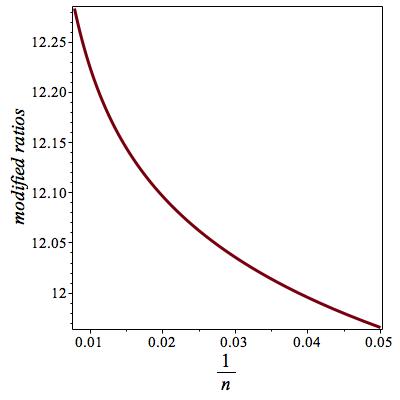} 
   \caption{ The first 128 modified ratios of the Navas-Brin group $B$ vs. $1/{n}.$ }
   \label{fig:modratb}
\end{minipage}
\end{figure}

In a similar fashion, using 4th order DAs, we were able to get 200 extra ratios for the 32-term series for Thompson group $F.$ The maximum error (as estimated by 1.5 s.d. of the DAs) is 1 part in $4 \times 10^{-5}$, which is graphically imperceptible. 
In the Appendix we give the (predicted) next 200 ratios, and their standard deviations. 

\begin{table}
   \centering
   \begin{tabular}{@{}lcc @{}} 
      \hline    
$k$ &      Actual error    & 1 standard deviation\\
\hline
1 & $2.69\times 10^{-17}$ & $2.02\times 10^{-17}$ \\
5 & $1.14\times 10^{-13}$ & $7.85\times 10^{-14}$ \\
10 & $3.37\times 10^{-11}$ & $2.08\times 10^{-11}$ \\
20 & $2.22\times 10^{-8}$ & $1.23\times 10^{-8}$ \\
30 & $9.63\times 10^{-7}$ & $5.39\times 10^{-7}$ \\
40 & $1.22\times 10^{-5}$ & $6.88\times 10^{-6}$ \\
50 & $7.59\times 10^{-5}$ & $4.73\times 10^{-5}$ \\
60 & $3.13\times 10^{-4}$ & $2.23\times 10^{-4}$ \\
70 & $9.39\times 10^{-4}$ & $8.11\times 10^{-4}$ \\
80 & $2.44\times 10^{-3}$ & $2.44\times 10^{-3}$ \\
89 & $4.63\times 10^{-3}$ & $5.38\times 10^{-3}$ \\
\hline
      
      \hline    
   \end{tabular}
   \vspace{2mm}
   \caption{Actual error in coefficient  O$(z^{31+k})$ and 1 standard deviation from the mean of the estimated coefficient.}
   \label{tab:llerror}
\end{table}

\section{Analysis of the Navas-Brin group $B.$}\label{sec:Brin}
This is an amenable group introduced independently by Navas \cite{N04} and Brin \cite{B05}, so we call it the Navas-Brin group $B,$ and is defined in subsection \ref{subsec:nb}. 
It has 2 generators, so the growth rate of the cogrowth sequence is 16. We gave a polynomial-time algorithm to generate the coefficients above, and have used this to generate 128 terms of the co-growth series.
We then used the method of series extension, described above, to give a further 590 ratios, the last of which we expect to be accurate to 1 part in $5 \times 10^{-7},$ while all earlier ratios will have a lower associated error.
We first show a plot of the modified ratios (\ref{mod-rat}) against $1/n$ in Figure \ref{fig:modratb}. Even if we knew nothing about the asymptotics of this group, the curvature of this plot provides strong evidence for a sub-exponential term, and we have proved that it cannot be a regular stretched-exponential term.

That is to say, the asymptotics for this series must grow more slowly than $$c_n \sim c \cdot \mu^n \cdot \kappa^{n^\sigma} \cdot n^g,$$ where $\mu=16,$ $0 < \sigma< 1,$ and $0 < \kappa < 1.$ Possible behaviour might be
$$c_n \sim c \cdot \mu^n \cdot \kappa^{n/\log{n}} \cdot n^g,$$ corresponding to a {\em numerical} value $ \sigma=1,$ which of course hides the logarithmic component.

In that case the ratios will be
$$r_n = \frac{c_n}{c_{n-1}} \sim \mu \left ( 1 + \frac{constant}{\log{n}} + \frac{g}{n} + \cdots \right ).$$ 
Note that we do not insist the the first correction term is $O(1/\log{n}),$ it could be a power of a logarithm, or some  other weakly decreasing function, but it cannot have a power-law  increase. For our purposes it suffices to take this term to be $O(1/\log{n}).$
We show the modified ratios (this gets rid of the $O(1/n)$ term in the asymptotics) in Figures \ref{fig:brinm1} and \ref{fig:brinm2} which are the same plot, but the first uses only the 128 exact coefficients, while the second uses the exact plus predicted ratios. From the first plot, it is clear that it would be an article of faith that the locus is going to 16 as $n \to \infty.$ By contrast, the second plot makes this conclusion far more plausible.

\begin{figure}[h!]
\setlength{\captionindent}{0pt}
\begin{minipage}{0.48\textwidth}
   \includegraphics[width=0.97\linewidth]{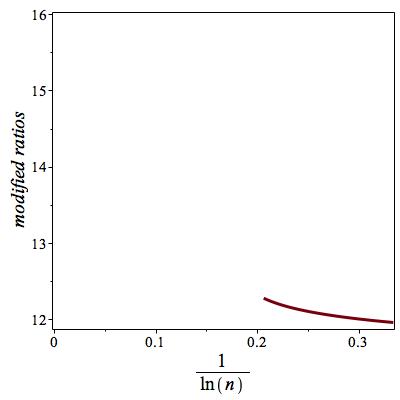} 
   \caption{The first 128 modified ratios for the Navas-Brin group $B$ vs. $1/\log{n}.$}
   \label{fig:brinm1}
\end{minipage}\hfill
\begin{minipage}{0.48\textwidth}
   \includegraphics[width=0.97\linewidth]{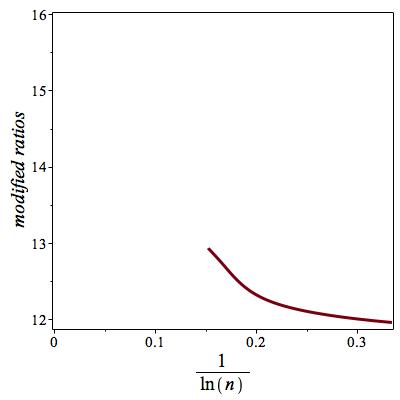} 
   \caption{ The first 718 modified ratios for the Navas-Brin group $B$ vs. $1/\log{n}.$ }
   \label{fig:brinm2}
\end{minipage}
\end{figure}

We next try and estimate the exponent $\sigma,$ which should be 1, without assuming $\mu=16.$ We use the method described below equation (\ref{rratio1}). With the 128 known terms, the estimators of $\sigma-2$ are shown in Figure \ref{fig:brins128}  and  show no evidence of approaching the expected value of $-1.$ If however we use twice as many terms, so using the next 128 predicted ratios, we get the plot shown in Figure \ref{fig:brins256}, which {\em is} plausibly approaching $-1.$ 

This highlights the value of numerically predicting further terms wherever possible.

\begin{figure}[h!]
\setlength{\captionindent}{0pt}
\begin{minipage}{0.48\textwidth}
   \includegraphics[width=0.97\linewidth]{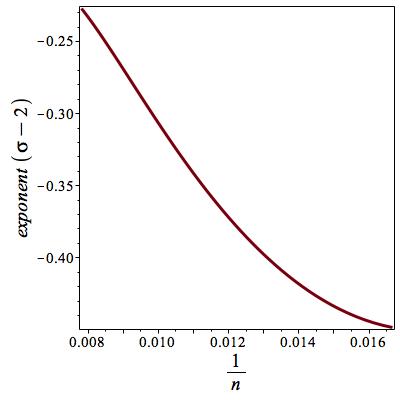} 
   \caption{Estimates of $\sigma-2$ from 128 terms of the Navas-Brin group $B.$ }
   \label{fig:brins128}
\end{minipage}\hfill
\begin{minipage}{0.48\textwidth}
   \includegraphics[width=0.97\linewidth]{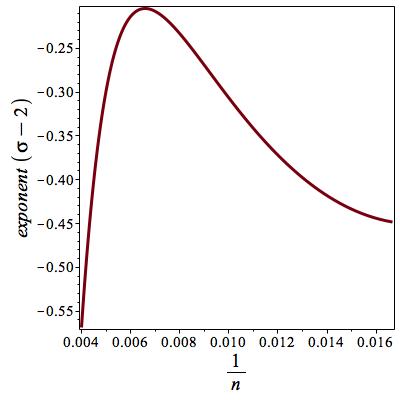} 
   \caption{Estimates of $\sigma-2$ from 256 terms of the Navas-Brin group $B.$ }
   \label{fig:brins256}
\end{minipage}
\end{figure}

\section{Analysis of Thompson's group $F$}\label{sec:Thompson}
For Thompson's group $F$ it is known that the series grows exponentially like $\mu^n.$ If $\mu = 16,$ the group is amenable. If it is amenable, there cannot be a sub-dominant term of the form $\kappa^{n^{\sigma}}$ with $0 < \sigma < 1,$ because the group contains the wreath products
$ {\mathbb Z}\wr {\mathbb Z}\wr {\mathbb Z}\wr \cdots \wr {\mathbb Z}$ as subgroups. This is a consequence of Theorem 1.3 in \cite{PS-C00} and results in \cite{PS-C02}, and is proved as Theorem \ref{thm:3.2} in Section \ref{sec:thom}.

We first study the modified ratios, defined by (\ref{mod-rat}). The modified ratio plot against $1/n$ is shown in figure \ref{fig:thomr1} and displays considerable curvature. By contrast, the same data plotted against $n^{-1/5},$ and shown in figure \ref{fig:thomr2} shows curvature in the opposite direction. This is strong evidence for the presence of a conventional stretched-exponential term of the sort we have seen in our study of the lamplighter group and the family $W_d.$  As mentioned above, the presence of such a term is incompatible with amenability. This is our first piece of evidence that the group is not amenable. Note too that this is quite different to the behaviour observed for the coefficients of the Navas-Brin group $B.$

\begin{figure}[h!]
\setlength{\captionindent}{0pt}
\begin{minipage}{0.48\textwidth}
   \includegraphics[width=0.97\linewidth]{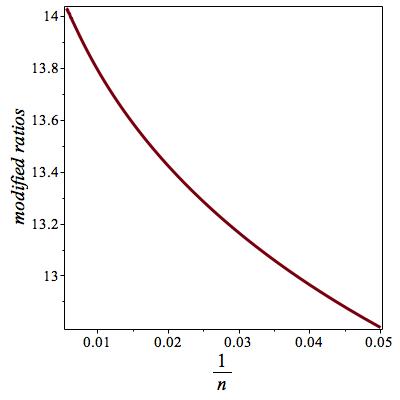} 
   \caption{Modified ratios vs. $1/n$ for Thompson's group $F.$ }
   \label{fig:thomr1}
\end{minipage}\hfill
\begin{minipage}{0.48\textwidth}
   \includegraphics[width=0.97\linewidth]{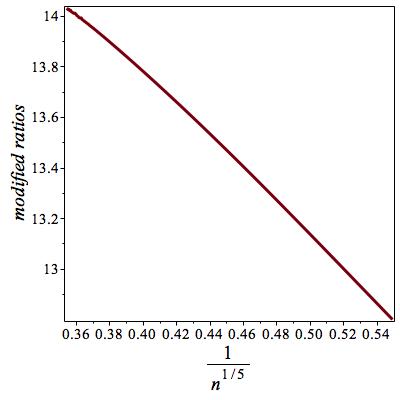} 
   \caption{Modified ratios vs. $n^{-1/5}$ for Thompson's group $F.$ }
   \label{fig:thomr2}
\end{minipage}
\end{figure}

In our subsequent analysis, we use both the exact coefficients and the extrapolated coefficients. While all extrapolated terms can be used in calculating the ratios, once one calculates first and second differences, errors are amplified, and so fewer terms can be used. That is why we quote the number of terms used for different calculations, as it is only to the quoted order that we are confident that the calculated quantities are accurate to graphical accuracy.

To estimate the exponents in the stretched-exponential term we use the procedure described in Section \ref{zsquared}, given by eqn. (\ref{eqn:tn}) and subsequent equations. This procedure allows for the presence of a confluent power of a logarithm, so that the stretched-exponential term is $\kappa^{n^{\sigma}\log^{\delta}{n}}.$ In this way, based on a series of length 80, we show plots of estimators of $2-\sigma$ and $-\delta$ in Figures \ref{fig:tsig1} and \ref{fig:tdel1}, plotted against $1/n.$ Extrapolating these, we estimate $\sigma \approx 1/2,$ and $\delta \approx 1/2.$ Recall that this is exactly the stretched-exponential behaviour of $(\mathbb{Z} \wr {\mathbb Z})\wr {\mathbb Z}.$

\begin{figure}[h!]
\setlength{\captionindent}{0pt}
\begin{minipage}{0.48\textwidth}
   \includegraphics[width=0.97\linewidth]{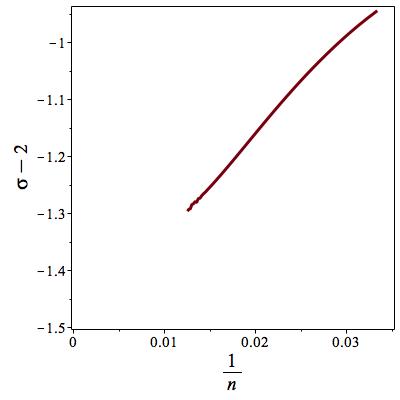} 
   \caption{Estimators of $\sigma-2$ for Thompson's group $F$ vs. $1/{n}.$ }
   \label{fig:tsig1}
\end{minipage}\hfill
\begin{minipage}{0.48\textwidth}
   \includegraphics[width=0.97\linewidth]{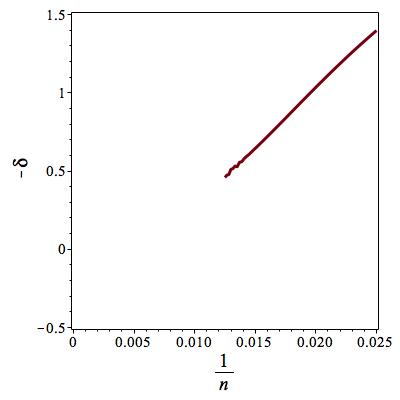} 
   \caption{Estimators of $-\delta$ for Thompson's group $F$ vs. $1/{n}.$  }
   \label{fig:tdel1}
\end{minipage}
\end{figure}

Reverting to the modified ratios, briefly discussed above, we plot these against $1/\sqrt{n}$ in figure \ref{fig:tmrat1}, using 186 terms. One observes that the plot still displays a little curvature, but in Figure \ref{fig:tmrat2} the plot of these same modified ratios against $\sqrt{\log{n}/n},$ is essentially linear. This is the appropriate power to extrapolate against, given our estimates of the stretched-exponential exponents. Extrapolating this to $n \to \infty$ we estimate the limit, which gives the growth constant, to be $14.8-15.1.$ This is well away from 16, which would be required for amenability.

\begin{figure}[h!]
\setlength{\captionindent}{0pt}
\begin{minipage}{0.48\textwidth}
   \includegraphics[width=0.97\linewidth]{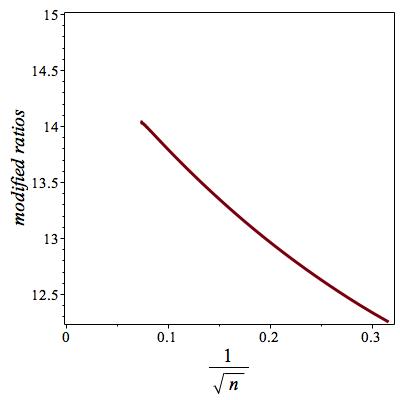} 
   \caption{The first 186 modified ratios for Thompson's group $F$ vs. $1/\sqrt{n}.$ }
   \label{fig:tmrat1}
\end{minipage}\hfill
\begin{minipage}{0.48\textwidth}
   \includegraphics[width=0.97\linewidth]{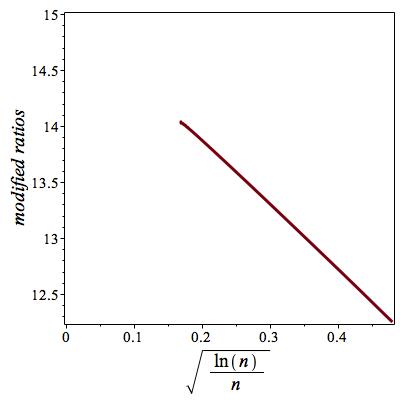} 
   \caption{The first 186 modified ratios for Thompson's group $F$ vs. $\sqrt{\log{n}/n}.$ }
   \label{fig:tmrat2}
\end{minipage}
\end{figure}

One simple test for amenability uses the fact that the ratio of successive coefficients asymptotes to the growth constant $\mu.$ For the lamplighter group, this ratio behaves as 
$$r^{(L)}_n = 9\left ( 1 + \frac{c}{n^{2/3}} + o\left (\frac{1}{n^{2/3}}\right ) \right ).$$   For ${\mathbb Z} \wr {\mathbb Z}$ one has  $$r^{(2)}_n = 16\left ( 1 + \frac{c\cdot {\log^{2/3}{n}}}{n^{2/3}} + o\left (\frac{\log^{2/3}{n}}{n^{2/3}}\right ) \right ),$$ and for the triple wreath product, $W_2,$  the corresponding result is $$r^{(3)}_n = 36\left ( 1 + \frac{c\sqrt{\log{n}}}{n^{1/2}} + o\left (\frac{\sqrt{\log{n}}}{n^{1/2}}\right ) \right ),$$ while for Thompson's group $F$ all we know is
$$r_n = \mu \left ( 1 + {\rm lower\,\,order\,\,terms} \right ),$$ where we suspect that the correction term is similar to that of the triple wreath product of ${\mathbb Z}.$
  
So, a simple test for amenability is to look at the three quotients $$\frac{9r_n}{16r^{(L)}_n},\,\,\, \frac{r_n}{r^{(2)}_n},\,\,\, {\rm and}\,\,\, \frac{4r_n}{9r^{(3)}_n}.$$ If Thompson's group $F$ is amenable, these quotients should all go to 1. In Figures \ref{fig:rtrl},  \ref{fig:rtrz2}, \ref{fig:rtrz3} we show these ratios plotted against $\sqrt{\log{n}/n},$ which is the appropriate power, though this choice is not critical. The ratios do not appear to be going to 1 in any of the three cases. For all cases we have used 200 ratios. To do this, we used the extended ratios for Thompson's group $F$ and also extended the ratios for $W_2$ from the known 132 ratios. Indeed, all three cases are consistent with a limit around $0.93 \pm 0.02,$ corresponding to $\mu = 14.9 \pm 0.3.$ This is entirely consistent with our previous estimate of $\mu \approx 15.0.$

\begin{figure}[h!]
\setlength{\captionindent}{0pt}
\begin{minipage}{0.3\textwidth}
\centerline {  \includegraphics[width=1.1\linewidth]{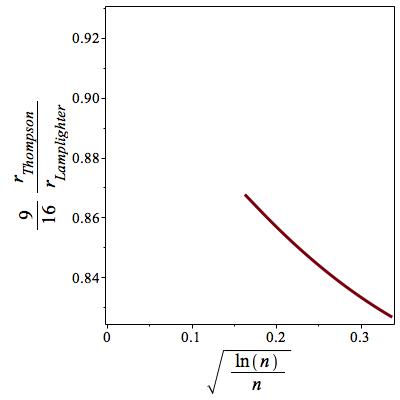}    }
   \caption{ Quotient of Thompson group and lamplighter group ratios using 200 terms.}
   \label{fig:rtrl}
\end{minipage}\hfill
\begin{minipage}{0.3\textwidth}
\centerline {  \includegraphics[width=1.1\linewidth]{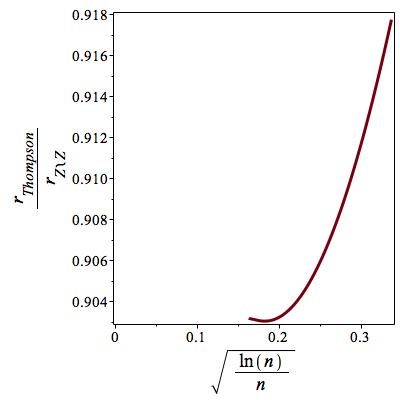}  }
  \caption{ Quotient of Thompson group and $\mathbb{Z} \wr {\mathbb Z}$ ratios using 200 terms.}
   \label{fig:rtrz2}
\end{minipage}\hfill
\begin{minipage}{0.3\textwidth}
\centerline {   \includegraphics[width=1.1\linewidth]{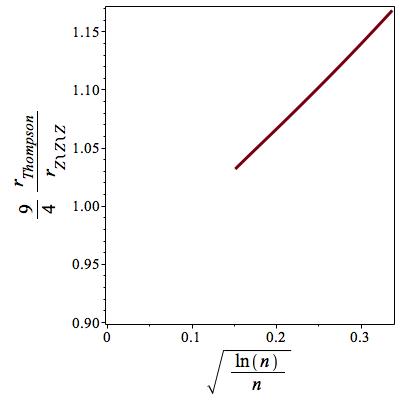}   }
  \caption{ Quotient of Thompson group and $(\mathbb{Z} \wr {\mathbb Z})\wr {\mathbb Z}$ ratios using 200 terms.}
   \label{fig:rtrz3}
\end{minipage}
\end{figure}

Finally, we take the approach of extrapolating the lower bounds produced in Section \ref{sec:moments}. Note that the sequence of bounds $\{b_n\}$ are bounds on $\sqrt{\mu}.$ We have no expectation as to how this sequence should approach its limit, so we first plot the bounds against $1/n$ in figure \ref{fig:tmrat3}. Some curvature is seen, which, as we have shown above, is evidence that the locus behaves as $$b_n \sim b_\infty \left ( 1 + \frac{c_1}{n^\alpha} + \frac{c_2}{n} + \cdots \right ).$$ We remove the term $O(1/n)$  in this case by forming the sequence $b^{(1)}_n = (n\cdot b_n - (n-2) \cdot b_{n-2})/2$ where we have shifted $n$ by 2 to remove the effect of a small odd-even oscillation if one shifts only by 1. We found that plotting $b^{(1)}_n$ against $1/\sqrt{n}$ gave a visually linear plot, and this is shown in figure \ref{fig:tmrat4}. Linearly extrapolating the last two entries gives the estimate $b_\infty \approx 3.875,$ so that $\mu \approx 15.02,$ in agreement with previous estimates above.

\begin{figure}[h!]
\setlength{\captionindent}{0pt}
\begin{minipage}{0.48\textwidth}
   \includegraphics[width=0.97\linewidth]{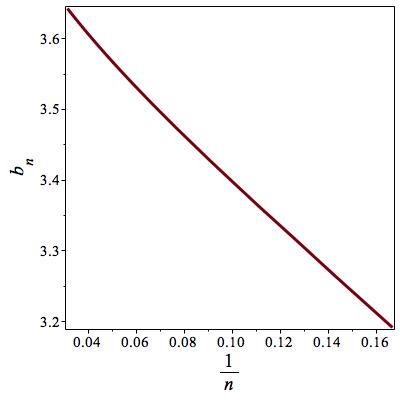} 
   \caption{Plot of bounds $b_n$ for Thompson's group $F$ against $1/n.$ }
   \label{fig:tmrat3}
\end{minipage}\hfill
\begin{minipage}{0.48\textwidth}
   \includegraphics[width=0.97\linewidth]{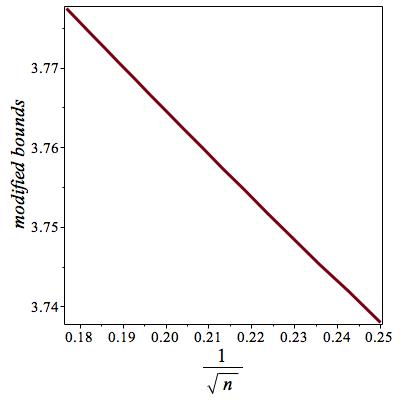} 
   \caption{Plot of modified bounds $b^{(1)}_n$ for Thompson's group $F$ against $1/\sqrt{n}.$ }
   \label{fig:tmrat4}
\end{minipage}
\end{figure}


\section{Conclusion}\label{sec:conclusion}

We have have given polynomial-time algorithms to generate terms of the cogrowth series for several groups. In particular, we have given the first series for the Navas-Brin group $B.$ We have also given an improved algorithm for the coefficients of Thompson's group $F,$ giving 32 terms of the cogrowth series, extending previous enumerations by 7 terms. We analysed these various series to develop numerical techniques to extract the asymptotics, and gave improved asymptotics for the Heisenberg group. We gave an improved lower bound on the growth-rate of the cogrowth series for Thompson's group $F,$ $\mu \ge 13.2693$ using the method in \cite{HHR15}. We generalised their method, showing that the cogrowth sequences for all these groups can be represented as the moments of a distribution. Extrapolation of the sequence of bounds suggests the limit is around 15.0, which is incompatible with amenability.

For Thompson's group $F$ we proved that, if the group is amenable, there cannot be a sub-dominant stretched exponential term in the asymptotics. The numerical data however provides compelling evidence for the presence of such a term. This observation suggests a potential path to a proof of non-amenability. 

We have extended the sequence of 32 terms for group $F$ by a further 200 terms (or, as appropriate, 200 ratios of successive terms), which we demonstrate are sufficiently accurate for the graphical approaches to analysis that we have taken.

A numerical study of the cogrowth sequence $c_n$ gives
$$ c_n \sim c \cdot \mu^n \cdot \kappa^{n^\sigma \log^\delta{n}} \cdot n^g,$$ where $\mu \approx 15,$ $\kappa \approx 1/e,$ $\sigma \approx 1/2,$ $\delta \approx 1/2,$ and $g \approx -1.$ The growth constant $\mu$ must be 16 for amenability. This estimate of the growth constant is the same as that obtained from the extrapolated bounds. These three approaches to the study of amenabilty lead us to the strong belief that Thompson's group $F$ is not amenable.

The difficulties we encountered in analysing ${\mathcal Z} \wr_{98} {\mathcal Z}$ and the Navas-Brin group $B$ does imply that there do exist groups whose cogrowth series are difficult to analyse. Nevertheless, in both those cases we were able to extract the correct asymptotics. Furthermore, the cogrowth series for Thompson's group $F$ did not behave like either of these two ``difficult'' groups, and indeed appeared to have a stretched exponential term with exponent values that were readily estimable. While we cannot rule out the presence of some previously unsuspected pathology in the asymptotic form, we believe that we have presented strong evidence for the belief that Thompson's group $F$ is not amenable

\section{Appendix}

Here are the (predicted) next 200 ratios for Thompson's group $F$. That is, the first ratio here is the coefficient of $z^{32}$ divided by the coefficient of $z^{31}.$ One standard deviation is $1.6 \times 10^{-21}$ for the first ratio, $8.4 \times 10^{-16}$ for the tenth ratio in this list, then $2.2 \times 10^{-13},\,\, 3.8 \times 10^{-11}, \,\, 8.3 \times 10^{-10}, \,\, 7.5 \times 10^{-9}, \,\, 3.1 \times 10^{-8}, \,\, 3.5 \times 10^{-7}, \,\, 1.2 \times 10^{-6}, \,\, 5.3 \times 10^{-6}, \,\, 3.3 \times 10^{-5}, \,\, 9.0 \times 10^{-5}, \,\, 1.3 \times 10^{-4}, $ for the twentieth, thirtieth, fourtieth, fiftieth, seventieth,  ninetieth and hunderd and tenthth, hundred and thirtieth, hundred and fiftieth, hundred and seventy-fifth and two hundredth ratios, respectively.\\ 
\newpage
\resizebox{1 \textwidth}{!}{
\scriptsize{

\begin{adjustbox}{angle=00}
    \begin{tabularhtx}{\textwidth}{\textheight}{|ccc|}
   \interrowfill
12.139382519134640546100910550116506& 12.169952350800835818835333877031972& 12.199326127345853916009149880943422\\  
12.227584513675824849745149961326117& 12.254800541517346423861272221078461& 12.281040527431431456883217428846113\\  
12.306364858371755791208652657890394& 12.330828666879163731631451465102197& 12.354482413853570301131793493786704\\ 
 12.377372393555580912478863352420963& 12.399541172868336566805611620482939& 12.421027974747801610713466511114207\\ 
 12.441869014094574198504147444796094& 12.462097792906309126366784632966479& 12.481745360450139667192011174605863\\ 
 12.500840543278316369259845344418039& 12.519410149156226556638850314986873& 12.537479148348791828180725386201064\\ 
 12.555070835196117213101528584412301& 12.572206972475257876321270606856264& 12.588907920692053314139061768312908\\
 12.605192754131413602669204643768679& 12.621079365262119276409607772782102& 12.636584558849435543137582196072181\\
 12.651724136967118435759923721500215& 12.666512975937091192133384667614540& 12.680965096098118923436248549296051\\
 12.695093725180061251425934047225685& 12.708911355958942268539084708497350& 12.722429798828036178804146087289492\\
 12.735660229810914461533323666877360& 12.748613234321123627064424451294341& 12.761298847540600696316471357098745\\
 12.773726591097905739200796016562265& 12.785905507056272163479082083260714& 12.797844188914176808112395397996995\\ 
 12.809550810208112073629601806420114& 12.821033151370240355220794516217934& 12.832298623805456848288841282862554\\ 
 12.843354291696056209817237940576090& 12.854206894770548244369354384398327& 12.864862864837654839427475941273709\\ 
 12.875328347113943574400714300355914& 12.885609214745162856554802222623966& 12.895711085169493589386718357886965\\ 
 12.905639332151085077301331779253254& 12.915399101706732471897030527017181& 12.924995323677800300649200885376963\\ 
 12.934432723922797241056261010730083& 12.943715839380256116094239676192280& 12.952849018349319276368496909220240\\ 
 12.961836439248190289410424756818502& 12.970682107319771206377319240620875& 12.979389881806587929393146535001536\\ 
 12.987963464694516602362349323441444& 12.996406405676872159827925022024219& 13.004722173676474514325001997839461\\ 
 13.012914070454692135046857458367114& 13.020985259779764800684267751152161& 13.028938801180381908385125175110548\\ 
 13.036777659579423368614400305996310& 13.044504679013068113321447729698093& 13.052122569524550318903068986491322\\ 
 13.059634066593784780607389266957567& 13.067041693422999439936815568787524& 13.074347966865955839456185038291272\\ 
 13.081555222866179498657087709047136& 13.088665732145507263332155324286416& 13.095681797656212049861838152969811\\ 
 13.102605564520384538741633073251136& 13.109439117936998695729698444879843& 13.116184350936469838871049112592397\\ 
 13.122843467661052218534876988958406& 13.129418244175823190427250023546405& 13.135910414942072482775413880736439\\ 
 13.142322234451185942341571946995161& 13.148654712658935599802876138058446& 13.154910018341609905494554223902216\\ 
 13.161089481701929123360329350075889& 13.167195007769142327138772890144629& 13.173227515058212755159290331997399\\ 
 13.179189107961895176661932566775513& 13.185080837993801098011934614407469& 13.190904563910270596599158822879997\\ 
 13.196660991623399487925223492990383& 13.202352151456208348963284966653968& 13.207979050670762502542808795553103\\ 
 13.213542543992015500848473483969209& 13.219044110867187297414288617379158& 13.224484921862603128812749940479344\\ 
 13.229866114857320555465055797943269& 13.235188321357468124261384585429120& 13.240453522986667992419868738673504\\ 
 13.245661794080940364573430988057957& 13.250815166741899371063964386311321& 13.255914147651533345355162750305551\\ 
 13.260959939896671980546840767426691& 13.265952758167346821144011231194731& 13.270895130906341393601998277628536\\ 
 13.275786005549024172548265823401768& 13.280627056720430964223829034834261& 13.285419125190171039873138415974114\\ 
 13.290161972217217703156676577608563& 13.294858653114477361638435828704210& 13.299508767603123209430266638744025\\ 
 13.304113073325880062519253728849004& 13.308670446314943572746000301996392& 13.313185140342015149024233508130128\\ 
 13.317658188282116335265357329472658& 13.322086115777917425146418299128619& 13.326471695739198888194127423197993\\ 
 13.330818453996063928243407490059108& 13.335121539230967797838900373066820& 13.339384163779883682357243658284554\\ 
 13.343606937084326659382677923429016& 13.347786320277756923575569080786819& 13.351930748190988263794674502771721\\ 
 13.356041865699128533066777923821716& 13.360110874363270214471235529540315& 13.364142777237566574247813605863778\\ 
 13.368138079655518061737105379768840& 13.372097273034426235102779701961991& 13.376020835131593723786616298972917\\ 
 13.379909230285796422796584282677189& 13.383762909644345312962700551534152& 13.387582311376027194727073026599142\\ 
 13.391391013985309598551994527199420& 13.395144666969365978147066024081715& 13.398865377039202234681959848510147\\ 
 13.402540990852287889254736126437580& 13.406196253710405879515140382321201& 13.409819665483104102788219559926390\\ 
 13.413411407143425386912606665853825& 13.416988156570367066227724319777381& 13.420539883769483957581690933051480\\ 
 13.424062672750839568072161252780645& 13.427535999385894978339537743097391& 13.430979699159296367928654845659964\\ 
 13.434394074827496721484051571897707& 13.437779418859588380431049420003005& 13.441136013513055381428134566110019\\ 
 13.444482843412927251038029918715121& 13.447782268650965022269207427038709& 13.451072044972393249466607491105591\\ 
 13.454316759036914691255964338022543& 13.457544856354882926610701767183908& 13.460737144623883340126835494564910\\ 
 13.463903116875215568913679808238264& 13.467042454042638170007261022024409& 13.470155360061132928954773338423225\\ 
 13.473213802552343969176414720413487& 13.476302647287223508754966768268084& 13.479337389227819997111748990692973\\ 
 13.482346421441562120101495301037430& 13.485293318045282335887952362450478& 13.488249148435842618350149753783486\\ 
 13.491179584354353752611550171676095& 13.494084747058831414056697557041621& 13.497011084300140602722763616549800\\ 
 13.499895368095523512785232116864023& 13.502701722883121030342927584362952& 13.505547791585498408392624616902183\\ 
 13.508333602964302346321586079889387& 13.511107842566379075799459233305947& 13.513936036377083173562688008219201\\ 
 13.516654351365134683610700727371316& 13.519431847277609941232036080607544& 13.522232264247358404187723455476999\\ 
 13.525000203908312007517604856476096& 13.527641315577516333811095751023520& 13.530300873952004160088528545275672\\ 
 13.532917871456062697723758553984577& 13.535493461803957879581733368101521& 13.538046382601181802194312557622836\\ 
 13.540554273647847830847507251037425& 13.543084323366200838300039176315332& 13.545599058351387123157120673148797\\ 
 13.548063007209007082860967768914526& 13.550504412576211288095020257342519& 13.552923260697021481475229482726246\\ 
 13.555319528240038802193441194513629& 13.557762668158072770433170818766971& 13.560302341521164494357743132135167\\ 
 13.562641374331046546303564833879801& 13.564958044304722734009028672257896& 13.567252292374466826596567631465810\\ 
 13.569777875695742862069575716786180& 13.572043884450805383578987603438786& 13.574288148599622255870938533612831\\ 
 13.576510610487748479571836167835301& 13.578837748814414710933467280925295& 13.581263278311749328511560181131020\\ 
 13.584419560999735066923148194498099& 13.586612599432547359359203779080153& 13.588785775454742632035040598882148\\ 
 13.590831265318841752277570171711095& 13.592959911744010496809257658001969& 13.594915982519419341022644469998104\\ 
 13.596998211959713327539754998388164& 13.599059952154497874738089781219857&  \\ 

     \end{tabularhtx}
   
\end{adjustbox}

}
}
\normalsize
\newpage

\section{Acknowledgements}
We wish to thank Andrew Rechnitzer for many stimulating discussions on this topic, and Murray Elder for helpful comments on the manuscript. AJG wishes to thank Nathan Clisby for his vastly superior version of the program to use differential approximants to predict further terms and ratios.  AEP wishes to thank ACEMS for financial support through a PhD top-up scholarship.

\end{document}